\documentclass[11pt,a4paper]{article}
\usepackage{mathrsfs}
\usepackage{multirow}
 \usepackage{epsfig}
 \usepackage{epstopdf}
\usepackage[T1]{fontenc}
\usepackage{geometry}
\usepackage{amsmath}
\usepackage{bm}
\usepackage{amsbsy,latexsym,amsfonts, epsfig, color, authblk, amssymb, graphics, bm}
\usepackage{epsf,slidesec,epic,eepic}
\usepackage{fancybox}
\usepackage{fancyhdr}
\usepackage{setspace}
\usepackage{cases}
\usepackage{nccmath}
\usepackage{url}
\usepackage[colorlinks, citecolor=blue]{hyperref}

\newtheorem{theorem}{Theorem}[section]
\newtheorem{lemma}{Lemma}[section]

\newtheorem{definition}{Definition}[section]

\newtheorem{proposition}{Proposition}[section]

\newtheorem{remark}{Remark}[section]
\newtheorem{property}{Property}

\newtheorem{aproposition}{Proposition}

\newenvironment{proof}{{\noindent \bf Proof:}}{\hfill$\Box$\medskip}
\definecolor{lred}{rgb}{1,0.8,0.8}
\definecolor{lblue}{rgb}{0.8,0.8,1}
\definecolor{dred}{rgb}{0.6,0,0}
\definecolor{dblue}{rgb}{0,0,0.5}
\definecolor{dgreen}{rgb}{0,0.5,0.5}

\title{Strong calmness of perturbed KKT system for
 a class of conic programming with degenerate solutions
 \footnote{Supported by the National Natural Science Foundation of China under project No.11571120 and the Natural Science Foundation of Guangdong Province under project No.2015A030313214.}}
 \author{Yulan Liu\footnote{School of Applied Mathematics, Guangdong University of Technology(ylliu@gdut.edu.cn.}\ \ {\rm and}\ \
 Shaohua Pan\footnote{School of Mathematics, South China University of Technology, Guangzhou (shhpan@scut.edu.cn).}}

\begin{document}

 \maketitle

 \begin{abstract}
  This paper is concerned with the strong calmness of the KKT solution mapping
  for a class of canonically perturbed conic  programming, which plays a central
  role in achieving fast convergence under situations when the Lagrange multiplier
  associated to a solution of these conic optimization problems is not unique.
  We show that the strong calmness of the KKT solution mapping is equivalent to
  a local error bound for solutions of perturbed KKT system, and is also
  equivalent to the pseudo-isolated calmness of the stationary point mapping
  along with the calmness of the multiplier set map at the corresponding reference point.
  Sufficient conditions are also provided for the strong calmness by establishing
  the pseudo-isolated calmness of the stationary point mapping in terms of
  the noncriticality of the associated multiplier, and the calmness of
  the multiplier set mapping in terms of a relative interior condition for
  the multiplier set. These results cover and extend the existing ones in
  \cite{Hager99,Izmailov12} for nonlinear programming and in \cite{Cui16,Zhang17}
  for semidefinite programming.
  \end{abstract}



 \section{Introduction}\label{sec1}

 Let $\mathbb{X}$ and $\mathbb{Y}$ be two finite dimensional real vector spaces
 equipped with an inner product $\langle \cdot,\cdot\rangle$ and its induced norm $\|\cdot\|$.
 Let $f\!:\mathbb{X}\to\mathbb{R}$ and $g\!:\mathbb{X}\to\mathbb{Y}$ be twice
 continuously differentiable functions. We consider the canonically
 perturbed optimization problem
 \begin{equation}\label{prob}
  \min_{x\in\mathbb{X}}\big\{f(x)-\langle a,x\rangle\!:\ g(x)-b\in \mathcal{K}\big\},
 \end{equation}
 where $(a,b)\in\mathbb{X}\times\mathbb{Y}$ is the perturbation parameter
 and $\mathcal{K}\subseteq\!\mathbb{Y}$ is a nonempty closed convex set.
 Throughout this paper, we assume that $\mathcal{K}$ is $C^2$-cone reducible
 (see Definition \ref{cone-reduce}), which covers all polyhedral convex sets
 and several classes of important non-polyhedral convex cones such as the positive
 semidefinite cone \cite[Corollary 4.6]{BCS99}, the second-order cone \cite[Lemma 15]{BR05}
 and the epigraph cone of the Ky Fan matrix $k$-norm \cite{DingST14}.

 \medskip

 Let $L\!:\mathbb{X}\times\mathbb{Y}\to\mathbb{R}$ be the Lagrange function of problem \eqref{prob}
 without perturbation:
 \[
   L(x,\lambda):=f(x)+\langle \lambda,g(x)\rangle\quad\ \forall(x,\lambda)\in\mathbb{X}\times\mathbb{Y}.
 \]
 For a given perturbation $(a,b)$, the KKT optimality condition for \eqref{prob} takes the form of
  \begin{equation}\label{KKT}
  \left\{\begin{array}{ll}
  \nabla_{\!x}L(x,\lambda)=a;\\
  \lambda\in \mathcal{N}_{\mathcal{K}}(g(x)-b)
  \end{array}\right.\Longleftrightarrow
  \left\{\begin{array}{ll}
  \nabla\!f(x)+\nabla\!g(x)\lambda=a;\\
  g(x)-b=\Pi_{\mathcal{K}}(g(x)-b+\lambda)
  \end{array}\right.
 \end{equation}
 where $\mathcal{N}_{\mathcal{K}}(z)$ denotes the normal cone of $\mathcal{K}$ at $z$
 in the sense of convex analysis \cite{Roc70}, $\Pi_{\mathcal{K}}$ means the projection operator
 onto $\mathcal{K}$, and for any given $\lambda\in\mathbb{Y}$, $\nabla_{\!x}L(\cdot,\lambda)$ is
 the adjoint of $L_x'(\cdot,\lambda)$, the derivative of $L(\cdot,\lambda)$ at $x\in\mathbb{X}$.
 In this paper, for a twice continuously differentiable $h\!:\mathbb{X}\to\mathbb{Y}$,
 we denote by $h'(x)$ the first-order derivative of $h$ at $x$, by $\nabla h(x)$
 the adjoint of $h'(x)$, and by $h''(x)$ the second-order derivative of $h$ at $x$.
 We define the KKT solution mapping $\mathcal{S}_{\rm KKT}\!:\mathbb{X}\times\mathbb{Y}
 \rightrightarrows\mathbb{X}\times\mathbb{Y}$, the stationary point map
 $\mathcal{X}_{\rm KKT}\!:\mathbb{X}\times\mathbb{Y}\rightrightarrows\mathbb{X}$,
 and the multiplier set map $\mathcal{M}\!:\mathbb{X}\times\mathbb{X}\times
 \mathbb{Y}\rightrightarrows\mathbb{Y}$ respectively by
 \begin{align}{}\label{SKKT}
  \mathcal{S}_{\rm KKT}(a,b):=\big\{(x,\lambda)\in\mathbb{X}\times\mathbb{Y}\ |\
  \nabla_{\!x}L(x,\lambda)=a,\lambda\in \mathcal{N}_{\mathcal{K}}(g(x)-b)\big\},\\
  \label{spoint}
  \mathcal{X}_{\rm KKT}(a,b):=\big\{x\in\mathbb{X}\ |\ \exists\,\lambda\in\mathbb{Y}\ {\rm such\ that\ system}\ \eqref{KKT}\ {\rm holds\ at}\ (x,\lambda)\big\},\\
  \mathcal{M}(x,a,b):=\big\{\lambda\in\mathbb{Y}\ |\  (x,\lambda)\in\mathcal{S}_{\rm KKT}(a,b)\big\}.\qquad\qquad\qquad
 \end{align}
 In addition, in the sequel we also need a counterpart of the multifunction $\mathcal{M}$, defined by
 \begin{equation}\label{MX}
  \mathcal{X}(\lambda,a,b):=\big\{x\in\mathbb{X}\ |\ (x,\lambda)\in\mathcal{S}_{\rm KKT}(a,b)\big\}.
 \end{equation}
  Clearly, the multifunction $\mathcal{X}$ can be regarded as
  a localization version of $\mathcal{X}_{\rm KKT}$.

 \medskip

 This work is mainly concerned with the strong calmness of $\mathcal{S}_{\rm KKT}$
 at a reference point $(\overline{x},\overline{\lambda})\in\mathcal{S}_{\rm KKT}(0,0)$
 with $\mathcal{M}(\overline{x},0,0)\ne\{\overline{\lambda}\}$, that is,
 $\overline{x}$ is a degenerate stationary point of the problem \eqref{prob}
 with $(a,b)=(0,0)$. The strong calmness is formally defined as follows.
  \begin{definition}\label{Upper-Lip}
  Let $(\overline{x},\overline{\lambda})$ be a KKT point of the problem \eqref{prob}
  with $(a,b)=(0,0)$ and $\mathcal{M}(\overline{x},0,0)\ne\{\overline{\lambda}\}$.
  The multifunction $\mathcal{S}_{\rm KKT}$ is said to have the strong calmness
  at the origin for $(\overline{x},\overline{\lambda})$
  if there exist $\delta>0,\varepsilon>0$ and a constant $\kappa>0$
  such that for any $(a,b)\in\mathbb{B}_{\delta}((0,0))$ and any
  $(x,\lambda)\in\mathcal{S}_{\rm KKT}(a,b)\cap\mathbb{B}_{\varepsilon}((\overline{x},\overline{\lambda}))$,
  the following estimate holds:
  \[
    \|x-\overline{x}\|+{\rm dist}(\lambda,\mathcal{M}(\overline{x},0,0))\le\kappa\|(a,b)\|.
  \]
 \end{definition}
 This property, different from the locally upper Lipschitz introduced
 by Robinson \cite{Robinson81} for a multifunction,
 is weaker than the isolated calmness of $\mathcal{S}_{\rm KKT}$
 but stronger than its calmness. Moreover, it does not imply
 the isolatedness of the stationary point $\overline{x}$.
 For the definition of (isolated) calmness, the reader may
 refer to \cite{DR09} or Section \ref{sec2}.
 When $\mathcal{M}(\overline{x},0,0)=\{\overline{\lambda}\}$,
 this property becomes the isolated calmness of $\mathcal{S}_{\rm KKT}$
 at the origin.

 \medskip

 To the best of our knowledge, the strong calmness in Definition \ref{Upper-Lip}
 was first introduced by Fern\'{a}ndez and Solodov \cite{Fernandez10} to guarantee
 the calmness of the KKT solution mapping for the canonically perturbed nonlinear
 programming, and then obtain the superlinear convergence of the stabilized
 sequential quadratic programming method (sSQP) by invoking \cite[Theorem 1]{Fischer02}.
 Later, Izmailov and Solodov \cite{Izmailov12,Izmailov13} provided some
 equivalent characterizations for the strong calmness of the KKT solution mapping
 in the setting of polyhedral conic optimization. They showed that
 this upper Lipschitz stability is not only equivalent to a local error bound
 for solutions of perturbed KKT system, but also equivalent to the noncriticality
 of the associated multiplier.
 As discussed in \cite{Wright98,Hager99,Fernandez10,Izmailov12,Izmailov13,Izmailov15},
 the strong calmness of the KKT solution map or equivalently
 the noncriticality of the associated multiplier is the key to achieve a fast
 convergence rate for the sSQP method or the augmented Lagrangian method (ALM)
 of the polyhedral conic optimization problems with degenerate solutions,
 i.e., the solutions with multiple Lagrange multipliers.
 Then, it is natural to ask whether these characterizations hold or not
 for nonpolyhedral conic optimization. If not, what conditions are enough?

 \medskip

 The main contribution of this work is to provide an affirmative answer to this question.
 Specifically, in Section \ref{sec3} we show that the strong calmness of
 $\mathcal{S}_{\rm KKT}$ at the origin for $(\overline{x},\overline{\lambda})
 \in\mathcal{S}_{\rm KKT}(0,0)$ is equivalent to the local error bound stated in
 Property \ref{Error-bound} for solutions of perturbed KKT system, and is also
 equivalent to the pseudo-isolated calmness (see Definition \ref{Pseudo-icalm})
 of the stationary point map $\mathcal{X}_{\rm KKT}$ at the origin for $\overline{x}$
 together with the calmness of the multiplier set map $\mathcal{M}$ at $(\overline{x},0,0)$
 for $\overline{\lambda}$. Among others, the calmness of $\mathcal{M}$ at $(\overline{x},0,0)$
 for $\overline{\lambda}$ is very weak and holds automatically in the polyhedral setting.
 However, unlike the polyhedral case, the noncriticality of the Lagrange multiplier
 $\overline{\lambda}$ is only necessary but not sufficient for the pseudo-isolated
 calmness of $\mathcal{X}_{\rm KKT}$. In Section \ref{sec4}, we show that,
 under some restrictions on the mapping $\nabla g(\overline{x})\!:\mathbb{X}\to\mathbb{Y}$,
 the noncriticality of the multiplier $\overline{\lambda}$ can
 guarantee the pseudo-isolated calmness of $\mathcal{X}_{\rm KKT}$,
 and consequently, some sufficient characterizations for the strong calmness
 of $\mathcal{S}_{\rm KKT}$ are obtained.

 \medskip

 Notice that for structured convex semidefinite programming with
 multiple solutions (actually degenerate solutions), Cui, Sun and Toh
 \cite{Cui16} have studied the calmness of the stationary point mapping
 and the strong calmness of the perturbed KKT system so as to
 capture the fast convergence of the ALM for such problems.
 Recently, for general nonlinear semidefinite programming,
 Zhang and Zhang \cite{Zhang17} established the strong calmness
 of $\mathcal{S}_{\rm KKT}$ under the noncriticality of the Lagrange
 multiplier along with some additional conditions, which improves
 the result in \cite{Cui16} for the perturbed KKT system since
 the noncriticality of the multiplier is much weaker than
 the second-order sufficient condition in \cite{Cui16} associated to
 a specified multiplier. This paper is to a great extent motivated by
 their works. We extend their results to a class of conic programming
 and remove the condition (i) of \cite[Theorem 3.3]{Zhang17}
 under the relative interior condition of multiplier set.

 \medskip

 In the rest of this paper, for a finite dimensional vector space $\mathbb{Z}$
 equipped with a norm $\|\cdot\|$, $\mathbb{B}_{\mathbb{Z}}$ denotes the closed
 unit ball centered at the origin in $\mathbb{Z}$ and for a given $z\in\mathbb{Z}$,
 $\mathbb{B}_{\delta}(z)$ means the closed ball of radius $\delta$ centered at $z$.
 For a given closed set $\Omega\subseteq\mathbb{Z}$,
 $\Pi_{\Omega}(\cdot)$ denotes the projection operator onto $\Omega$,
 and ${\rm dist}(x,\Omega):=\inf_{z\in\Omega}\|z-x\|$ for a given $x\in\mathbb{Z}$
 means the distance of $x$ from the set $\Omega$; and for a given
 nonempty convex cone $K\subseteq\mathbb{Y}$, $K^{\circ}$ denotes
 the negative polar of $K$. For a linear map $\mathcal{A}$,
 $\mathcal{A}^*$ denotes the adjoint of $\mathcal{A}$.
 \section{Preliminaries}\label{sec2}

 Throughout this section, $\mathbb{Z}$ and $\mathbb{W}$ denote two finite
 dimensional vector spaces equipped with the inner product $\langle\cdot,\cdot\rangle$
 and its induced norm $\|\cdot\|$. We first recall some background knowledge
 from the excellent monographs \cite{RW98,BS00,KK02,Mordu06,DR09}.
 Let $\Omega\subseteq\mathbb{Z}$ be a nonempty set. Fix an arbitrary
 $\overline{z}\in\Omega$. The contingent cone to $\Omega$ at $\overline{z}$ is defined as
 \[
   \mathcal{T}_{\Omega}(\overline{z}):=\big\{w\in\mathbb{Z}\ |\ \exists t_k\downarrow 0,\,
   w^k\to w\ \ {\rm as}\ k\to\infty\ {\rm with}\ \overline{x}+t_kw_k\in\Omega\big\},
 \]
 while the basic/limiting normal cone to $\Omega$ at $\overline{z}$ admits the following
 representation
 \[
   \mathcal{N}_\Omega(\overline{z})=\limsup_{z\xrightarrow[\Omega]{}\overline{z}}\widehat{\mathcal{N}}_{\Omega}(z)
   \ \ {\rm with}\ \
   \widehat{\mathcal{N}}_{\Omega}(z)
   :=\Big\{v\in\mathbb{Z}\ |\ \limsup_{z\xrightarrow[\Omega]{}\overline{z}}
   \frac{\langle v,z-\overline{z}\rangle}{\|z-\overline{z}\|}\le 0\Big\}.
 \]
 When $\Omega$ is locally closed around $\overline{z}\in\Omega$, this is equivalent to
 the original definition by Mordukhovich \cite{Mordu76}, i.e.,
 \(
   \mathcal{N}_\Omega(\overline{z}):=\limsup_{z\to\overline{z}}\big[{\rm cone}(z-\Pi_{\Omega}(z))\big],
 \)
 and when $\Omega$ is convex, $\mathcal{N}_{\Omega}(\overline{z})$ becomes
 the normal cone in the sense of convex analysis \cite{Roc70}.

 \medskip

 Let $\mathcal{F}\!:\mathbb{Z}\rightrightarrows\mathbb{W}$ be a given multifunction.
 Consider an arbitrary $(\overline{z},\overline{w})\in{\rm gph}\mathcal{F}$
 such that $\mathcal{F}$ is locally closed at $(\overline{z},\overline{w})$,
 where ${\rm gph}\mathcal{F}$ denotes the graph of $\mathcal{F}$.
 We recall from \cite{RW98,DR09} the concepts of metric subregularity
  and calmness of the multifunction $\mathcal{F}$.
  \begin{definition}\label{subregular}
   The multifunction $\mathcal{F}\!:\mathbb{Z}\rightrightarrows\mathbb{W}$ is said to
   be metrically subregular at $\overline{z}$  for $\overline{w}\in\mathcal{F}(\overline{z})$
   if there exists $\kappa>0$ along with $\varepsilon>0$ and $\delta>0$ such that
   for all $z\in\mathbb{B}_{\varepsilon}(\overline{z})$,
   \[
     {\rm dist}\big(z,\mathcal{F}^{-1}(\overline{w})\big)
     \le \kappa\,{\rm dist}\big(\overline{w},\mathcal{F}(z)\cap\mathbb{B}_{\delta}(\overline{w})\big).
   \]
  \end{definition}
 \begin{definition}\label{calm-def}
  The multifunction $\mathcal{F}$ is said to be calm at $\overline{z}$ for $\overline{w}$
  if there exists $\kappa>0$ along with $\varepsilon>0$ and $\delta>0$ such that for all
  $z\in\mathbb{B}_{\varepsilon}(\overline{z})$,
  \begin{equation}\label{calm-inclusion1}
    \mathcal{F}(z)\cap\mathbb{B}_{\delta}(\overline{w})
    \subseteq\mathcal{F}(\overline{z})+\kappa\|z-\overline{z}\|\mathbb{B}_{\mathbb{W}}.
  \end{equation}
  If in addition $\mathcal{F}(\overline{z})=\{\overline{w}\}$, the multifunction
  $\mathcal{F}$ is said to be isolated calm at $\overline{z}$ for $\overline{w}$.
 \end{definition}

 By \cite[Exercise 3H.4]{DR09}, the neighborhood $\mathbb{B}_{\delta}(\overline{w})$ in
 Definition \ref{subregular} and the restriction of $z\in\mathbb{B}_{\varepsilon}(\overline{z})$
 in Definition \ref{calm-def} can be removed. The graphical derivative of $\mathcal{F}$
 is a convenient tool to study the isolated calmness of $\mathcal{F}$.
 Recall from \cite{Aubin81} the graphical derivative of $\mathcal{F}$ at $(\overline{z},\overline{w})$
 is the mapping $D\mathcal{F}(\overline{z}|\overline{w})\!:\mathbb{Z}\rightrightarrows\mathbb{W}$
 defined by $\Delta w\in D\mathcal{F}(\overline{z}|\overline{w})(\Delta z)$
 if and only if $(\Delta z,\Delta w)\in\mathcal{T}_{{\rm gph}\mathcal{F}}(\overline{z},\overline{w})$.
 With the graphical derivative of $\mathcal{F}$, the following result holds.
  \begin{lemma}\label{chara-icalm}(see \cite[Proposition 2.1]{KR92}
  or \cite[Proposition 4.1]{Levy96})\ The multifunction $\mathcal{F}$
  is isolated calm at $\overline{z}$ for $\overline{w}$
  if and only if $D\mathcal{F}(\overline{z}|\overline{w})(0)=\{0\}$.
 \end{lemma}

  Let $\overline{x}$ be a feasible point of the problem \eqref{prob} with $(a,b)\!=\!(0,0)$.
  The critical cone of \eqref{prob} with $(a,b)\!=\!(0,0)$ at $\overline{x}$ is defined as
  \(
    \mathcal{C}(\overline{x}):=\big\{d\in\mathbb{X}\ |\ g'(\overline{x})d\in\mathcal{T}_{\mathcal{K}}(g(\overline{x})),
    \langle \nabla\!f(\overline{x}),d\rangle\le 0\big\}.
  \)
  When $\overline{x}$ is a stationary point and $\overline{\lambda}\in\mathcal{M}(\overline{x},0,0)$,
  the critical cone can be rewritten as
  \[
    \mathcal{C}(\overline{x})=\big\{d\in\mathbb{X}\ |\ g'(\overline{x})d\in\mathcal{C}_{\mathcal{K}}(g(\overline{x}),\overline{\lambda})\big\}
  \]
 where, for any $y\in\mathcal{K}$,
 \(
   \mathcal{C}_{\mathcal{K}}(y,u)
 \)
 is the critical cone of $\mathcal{K}$ at $y$ w.r.t. $u\in\mathcal{N}_{\mathcal{K}}(y)$, defined by
 \[
    \mathcal{C}_{\mathcal{K}}(y,u):=\mathcal{T}_{\mathcal{K}}(y)\cap u^{\perp}.
 \]
 Motivated by \cite[Definition 3.2]{Zhang17}, we introduce
 the concept of noncritical multipliers.
 \begin{definition}\label{Def-noncritical}
  Let $\overline{x}$ be a stationary point of \eqref{prob} with
  $(a,b)=(0,0)$ and $\overline{\lambda}\in\mathcal{M}(\overline{x},0,0)$.
  The Lagrange multiplier $\overline{\lambda}$ is said to be noncritical
  if the following generalized equation
  \[
    0\in\nabla_{xx}^2L(\overline{x},\overline{\lambda})\xi
    +\nabla g(\overline{x})D\mathcal{N}_{\mathcal{K}}(g(\overline{x})|\overline{\lambda}))(g'(\overline{x})\xi)
  \]
  has the unique trivial solution $\xi=0$, and otherwise $\overline{\lambda}$
  is said to be critical.
 \end{definition}

  Now let us recall the definition of the $C^{2}$-cone reducibility for a closed convex set.
 \begin{definition}\label{cone-reduce}(\cite[Definition 3.135]{BS00})
   A closed convex set $\Omega$ in $\mathbb{Y}$ is said to be $C^{2}$-cone
   reducible at $\overline{y}\in\Omega$, if there exist an open neighborhood
   $\mathcal{Y}$ of $\overline{y}$, a pointed closed convex cone $D\subseteq\mathbb{Z}$
   and a twice continuously differentiable mapping $\Xi\!: \mathcal{Y}\to\mathbb{Z}$
   such that (i) $\Xi(\overline{y})=0$; (ii) $\Xi'(\overline{y})\!:\mathbb{Y}\to\mathbb{Z}$
   is onto; (iii) $\Omega\cap\mathcal{Y}=\{y\in\mathcal{Y}\ |\  \Xi(y)\in\!D\}$.
   We say that the closed convex set $\Omega$ is $C^2$-cone reducible
   if $\Omega$ is $C^2$-cone reducible at every $y\in \Omega$.
 \end{definition}
  Since $\mathcal{K}$ is assumed to be $C^{2}$-cone reducible in this paper,
  by \cite[Proposition 3.136]{BS00} the set $\mathcal{K}$ is second-order regular
  at each $y\in\mathcal{K}$, and hence $\mathcal{T}_{\mathcal{K}}^{i,2}(y,h)=
  \mathcal{T}_{\mathcal{K}}^{2}(y,h)$ at each $y\in\mathcal{K}$ for any $h\in\mathbb{Y}$,
  where $\mathcal{T}_{\mathcal{K}}^{i,2}(y,h)$ and $\mathcal{T}_{\mathcal{K}}^{2}(y,h)$
  denotes the inner and outer second order tangent sets to $\mathcal{K}$ at $y$
  in the direction $h\in\mathbb{Y}$. From the standard reduction approach,
  we have the following result on the representation of the normal cone of
  $\mathcal{K}$ and the ``sigma term'' of the $C^2$-cone reducible set
  $\mathcal{K}$ (see \cite[Equation(3.266)\&(3.274)]{BS00}).
  \begin{lemma}\label{lemma-reduction}
   Let $\overline{y}\in \mathcal{K}$ be given. There exist an open neighborhood
   $\mathcal{Y}$ of $\overline{y}$, a pointed closed convex cone $D\subseteq\mathbb{Z}$,
   and a twice continuously differentiable mapping $\Xi\!:\mathcal{Y}\to\mathbb{Z}$
   satisfying conditions (i)-(iii) in Definition \ref{cone-reduce} such that
   for all $y\in\mathcal{Y}$,
   \[
     \mathcal{N}_{\mathcal{K}}(y)=\nabla\Xi(y)\mathcal{N}_{D}(\Xi(y)).
   \]
   Also, for any $\overline{\lambda}\in\mathcal{N}_{\mathcal{K}}(\overline{y})$,
   there exists a unique $\overline{\mu}\in\!\mathcal{N}_{D}(\Xi(\overline{y}))$ such that
   $\overline{\lambda}=\nabla\Xi(\overline{y})\overline{\mu}$, and
   \begin{equation}\label{Upsilon}
     \Upsilon(h):=-\sigma\big(\overline{\lambda},\mathcal{T}_{\mathcal{K}}^2(\overline{y},h)\big)
     =\langle \overline{\mu},\Xi''(\overline{y})(h,h)\rangle
     \quad \forall h\in\mathcal{C}_{\mathcal{K}}(\overline{y},\overline{\lambda}).
   \end{equation}
  \end{lemma}

  Next we recall a useful result on the directional derivative of the projection
  operator $\Pi_{\mathcal{K}}$. Fix an arbitrary $y\in\mathbb{Y}$.
  Write $\overline{y}=\Pi_{\mathcal{K}}(y)$ and let
  $\overline{\lambda}\in\mathcal{N}_{\mathcal{K}}(\overline{y})$.
  Since $\mathcal{K}$ is $C^{2}$-cone reducible, by \cite[Theorem 7.2]{BCS99} the mapping
  $\Pi_{\mathcal{K}}$ is directionally differentiable at $y$ and the directional derivative
  $\Pi_{\mathcal{K}}'(y;h)$ for any direction $h\in\mathbb{Y}$ satisfies
  \[
    \Pi_{\mathcal{K}}'(y;h)=\mathop{\arg\min}_{d\in\mathbb{Y}}
    \big\{\|d-h\|^2-\sigma\big(\overline{\lambda},\mathcal{T}_{\mathcal{K}}^2(\overline{y},d)\big)
    \!: d\in\mathcal{C}_{\mathcal{K}}(\overline{y},\overline{\lambda})\big\}.
  \]
  In addition, by following the arguments as those for \cite[Theorem 3.1]{WZZhang14},
  one can obtain
  \[
    \mathcal{T}_{{\rm gph}\mathcal{N}_{\mathcal{K}}}(z,w)=
    \big\{(\Delta z,\Delta w)\in\mathbb{Y}\times\mathbb{Y}\ |\
     \Pi_{\mathcal{K}}'(z+w;\Delta z+\Delta w)=\Delta z\big\}.
  \]
  Combining this with \cite[Lemma 10]{DingSZ17}, we have the following conclusion
  for the graphical derivative of $\mathcal{N}_{\mathcal{K}}$,
  the directional derivative of $\Pi_{\mathcal{K}}$ and the critical cone of
  the set $\mathcal{K}$.
  \begin{lemma}\label{dir-proj}
   Let $z\in\!\mathbb{Y}$ be a given vector. Write $\overline{z}:=\Pi_{\mathcal{K}}(z)$
   and $\overline{\mu}:=z-\overline{z}$. Then,
   with $\Upsilon(\cdot)=-\sigma\big(\overline{\mu},\mathcal{T}_{\mathcal{K}}^2(y,\cdot)\big)
   =\langle u,\Xi''(\overline{z})(\cdot,\cdot)\rangle$ for $u\in\mathcal{N}_D(\Xi(\overline{z}))$,
   it holds that
   \begin{align}\label{system-regular}
    \Delta\lambda\in D\mathcal{N}_{\mathcal{K}}(\overline{z}|\overline{\mu})(\Delta y)
    &\Longleftrightarrow \Delta y-\Pi_{\mathcal{K}}'(z;\Delta y\!+\!\Delta\lambda)=0\nonumber\\
    &\Longleftrightarrow
    \left\{\begin{array}{ll}
      \Delta y\in\mathcal{C}_{\mathcal{K}}(\overline{z},\overline{\mu}),\\
      \Delta\lambda-\frac{1}{2}\nabla\Upsilon(\Delta y)\in[\mathcal{C}_{\mathcal{K}}(\overline{z},\overline{\mu})]^{\circ},\\
      \langle\Delta y,\Delta\lambda\rangle
     =-\sigma(\overline{\mu},\mathcal{T}_{\mathcal{K}}^2(\overline{z},\Delta y)).
     \end{array}\right.
   \end{align}
 \end{lemma}
  By using Lemma \ref{dir-proj} and Definition \ref{Def-noncritical},
  it is immediate to obtain the following result.
  \begin{proposition}\label{noncritical-prop}
   Let $(\overline{x},\overline{\lambda})$ be a KKT point of the problem \eqref{prob}
   with $(a,b)=(0,0)$. Then $\overline{\lambda}$ is noncritical iff
   all solutions of the following system have the form of $(0,*)\in\mathbb{X}\times\mathbb{Y}$:
  \begin{equation}\label{noncritical-system}
   \left\{\begin{array}{l}
    \nabla_{xx}^2L(\overline{x},\overline{\lambda})\xi+\nabla g(\overline{x})v=0,\\
    g'(\overline{x})\xi-\Pi_{\mathcal{K}}'(g(\overline{x})\!+\!\overline{\lambda};g'(\overline{x})\xi+v)=0.
    \end{array}\right.
   \end{equation}
  \end{proposition}

  To close this section, we study the calmness of a multifunction
  related to the perturbed KKT system. For a given $x\in\mathbb{X}$,
  we define the multifunction $\mathcal{G}_x\!:\mathbb{X}\times\mathbb{Y}
  \rightrightarrows\mathbb{Y}$ by
  \begin{equation}\label{MGmap}
   \mathcal{G}_{x}(\eta,y):=\big\{\lambda\in\mathbb{Y}\ |\ \eta + \nabla\!g(x)\lambda=0,\,
   y\!-\!\Pi_{\mathcal{K}}(y\!+\!\lambda)=0\big\}.
  \end{equation}
  Let $(\overline{x},\overline{\lambda})$ be a KKT point of \eqref{prob}
  with $(a,b)\!=\!(0,0)$. Clearly $\mathcal{G}_{\overline{x}}(\nabla\!f(\overline{x}),g(\overline{x}))
  =\mathcal{M}(\overline{\lambda},0,0)$. The following proposition gives
  some conditions for the calmness of $\mathcal{G}_{\overline{x}}$ at
  $(\nabla\!f(\overline{x}),g(\overline{x}))$ for $\overline{\lambda}$ which,
  as will be seen in Section \ref{sec3}, are sufficient for that of
  $\mathcal{M}$ at $(\overline{x},0,0)$ for $\overline{\lambda}$.
  \begin{proposition}\label{calm-MGx}
   Let $(\overline{x},\overline{\lambda})$ be a KKT point of the problem \eqref{prob}
   with $(a,b)=(0,0)$. Write $\overline{\eta}\!:=\nabla\!f(\overline{x})$
   and $\overline{y}\!:=g(\overline{x})$. Define the multifunction
   $\mathcal{E}(\eta):=\mathcal{H}(\eta)\cap\mathcal{N}_{\mathcal{K}}(\overline{y})$
   for $\eta\in\mathbb{X}$, where
   \(
    \mathcal{H}(\eta)\!:=\big\{\lambda\in\mathbb{Y}\ |\ \eta+\nabla\!g(\overline{x})\lambda=0\big\}.
   \)
   Then the following statements hold:
   \begin{itemize}
    \item [(a)] the multifunction $\mathcal{G}_{\overline{x}}$ is calm at $(\overline{\eta},\overline{y})$
                 for $\overline{\lambda}$ provided that $\mathcal{E}$ is calm at $\overline{\eta}$ for $\overline{\lambda}$;

    \item [(b)] $\mathcal{E}$ is calm at $\overline{\eta}$ for $\overline{\lambda}$
                if there exist $\delta>0$ and $\gamma>0$ such that for any $\lambda\in\mathbb{B}_{\delta}(\overline{\lambda})$,
                 \begin{equation}\label{assump-MQ1}
                  {\rm dist}(\lambda,\mathcal{E}(\overline{\eta}))
                  \le \gamma\max\big\{{\rm dist}(\lambda,\mathcal{N}_{\mathcal{K}}(\overline{y})),
                  {\rm dist}(\lambda,\mathcal{H}(\overline{\eta}))\big\},
                 \end{equation}
               which is particularly implied by the condition that
               ${\rm ri}(\mathcal{N}_{\mathcal{K}}(\overline{y}))\cap\mathcal{H}(\overline{\eta})\ne\emptyset$.
   \end{itemize}
  \end{proposition}
  \begin{proof}
  {\bf (a)} Notice that $\mathcal{G}_{\overline{x}}(\eta,y)
  =\mathcal{H}(\eta)\cap\mathcal{N}_{\mathcal{K}}(y)$
  for any $(\eta,y)\in\mathbb{X}\times\mathbb{Y}$. Since $\mathcal{K}$
  is assumed to be $C^2$-cone reducible, the multifunction $\mathcal{N}_{\mathcal{K}}$
  is calm at $\overline{y}$ for $\overline{\lambda}$ by \cite[Theorem 2.1]{LiuPan17}.
  Since the multifunction $\mathcal{H}$ is polyhedral,
  from \cite[Proposition 1]{Robinson81} we know that $\mathcal{H}$
  is calm at $\overline{\eta}$ for $\overline{\lambda}$.
  By \cite[Theorem 3.6]{KK02} and the given assumption that
  $\mathcal{E}$ is calm at $\overline{\eta}$ for $\overline{\lambda}$,
  it suffices to check that $\mathcal{H}^{-1}$ has the Aubin property
  at $\overline{\lambda}$ for $\overline{\eta}$. For this purpose,
  fix arbitrary $\lambda,\lambda'\in\mathbb{Y}$ and take an arbitrary
  $u\in\mathcal{H}^{-1}(\lambda)$.
  Then, we have $u+\nabla g(\overline{x})\lambda=0$.
  Notice that $\mathcal{H}^{-1}(\lambda')$ is a closed convex set.
  Let $u'\in \mathcal{H}^{-1}(\lambda')$ be such that
  \(
    \|u-u'\|={\rm dist}(u,\mathcal{H}^{-1}(\lambda')).
  \)
 Then, by using $u'+\nabla g(\overline{x})\lambda'=0$, it holds that
  \[
   {\rm dist}(u,\mathcal{H}^{-1}(\lambda'))=\|u-u'\|
   =\|-\nabla g(\overline{x})\lambda+\nabla g(\overline{x})\lambda'\|
   \le\|\nabla g(\overline{x})\|\|\lambda-\lambda'\|.
  \]
  This shows that $\mathcal{H}^{-1}$ is Lispchitz continuous in $\mathbb{X}$
  by \cite[Definition 9.26]{RW98}, and then has the Aubin property at
  $\overline{\lambda}$ for $\overline{\eta}$.
  The desired result (a) then follows.

  \medskip
  \noindent
  {\bf (b)} Assume that there exist $\delta>0$ and $\gamma>0$ such that
  for any $\lambda\in\mathbb{B}_{\delta}(\overline{\lambda})$,
  inequality \eqref{assump-MQ1} holds. Since $\mathcal{H}$ is calm
  at $\overline{\eta}$ for $\overline{\lambda}$,
  there exist $\widetilde{\delta}>0$ and $\widetilde{\kappa}>0$ such that
  \[
    {\rm dist}(\lambda,\mathcal{H}(\overline{\eta}))
     \le \widetilde{\kappa}{\rm dist}(\overline{\eta},\mathcal{H}^{-1}(\lambda))
     \quad\ \forall\lambda\in\mathbb{B}_{\widetilde{\delta}}(\overline{\lambda}).
  \]
  Set $\delta'=\min(\delta,\widetilde{\delta})$. Then, together with
  inequality \eqref{assump-MQ1}, for any
  $\lambda\in\mathbb{B}_{\delta'}(\overline{\lambda})$ we have
  \begin{equation}\label{assump-MQ}
     {\rm dist}(\lambda,\mathcal{E}(\overline{\eta}))
     \le \gamma\max(1,\widetilde{\kappa})
     \max\big\{{\rm dist}(\lambda,\mathcal{N}_{\mathcal{K}}(\overline{y})),
     {\rm dist}(\overline{\eta},\mathcal{H}^{-1}(\lambda))\big\},
  \end{equation}
  which implies that the following inequality holds
  \begin{equation}\label{calm-ineq}
    {\rm dist}(\lambda,\mathcal{E}(\overline{\eta}))
    \le \gamma\max(1,\widetilde{\kappa}){\rm dist}(\overline{\eta},\mathcal{H}^{-1}(\lambda))
    \quad\ \forall\lambda\in\mathcal{N}_{\mathcal{K}}(\overline{y})\cap\mathbb{B}_{\delta'}(\overline{\lambda}).
  \end{equation}
  Notice that $\mathcal{E}$ is calm at $\overline{\eta}$ for $\overline{\lambda}$
  if and only if $\mathcal{E}^{-1}$ is metrically subregular at $\overline{\lambda}$
  for $\overline{\eta}$, which is equivalent to requiring that there exist
  $\varepsilon>0$ and $\nu>0$ such that
  \begin{equation}\label{calm-ineq}
    {\rm dist}(\lambda,\mathcal{E}(\overline{\eta}))
    \le \nu{\rm dist}(\overline{\eta},\mathcal{H}^{-1}(\lambda))
    \quad\ \forall\lambda\in\mathcal{N}_{\mathcal{K}}(\overline{y})\cap\mathbb{B}_{\varepsilon}(\overline{\lambda}).
  \end{equation}
  This shows that the condition in \eqref{assump-MQ1} implies
  the calmness of $\mathcal{E}$ at $\overline{\eta}$ for $\overline{\lambda}$.
  While the condition in \eqref{assump-MQ1} is implied by
   ${\rm ri}(\mathcal{N}_{\mathcal{K}}(\overline{y}))
   \cap\mathcal{H}(\overline{\eta})\ne\emptyset$ by \cite[Corollary 3]{Bauschke99}.
  \end{proof}
  \begin{remark}
   From \cite[Section 3.1]{Ioffe08} the assumption in \eqref{assump-MQ1} is
   actually a metric qualification which is equivalent to the calmness of
   the following mapping at $(0,0)$ for $\overline{\lambda}$:
   \begin{equation}\label{Mmap}
    \mathcal{F}(u,v):=\big\{\lambda\in\mathbb{Y}\!: \lambda+u\in\mathcal{N}_{\mathcal{K}}(\overline{y}),\,
    \lambda+v\in\mathcal{H}(\overline{\eta})\big\}.
   \end{equation}
   Clearly, the metric qualification in \eqref{assump-MQ1} is weaker than
   the bounded linear regularity of the collection $\{\mathcal{N}_{\mathcal{K}}(\overline{y}),\mathcal{H}(\overline{\eta})\}$,
   while the latter is implied by  ${\rm ri}(\mathcal{N}_{\mathcal{K}}(\overline{y}))\cap\mathcal{H}(\overline{\eta})\ne\emptyset$.
  \end{remark}
  \section{Equivalent characterizations}\label{sec3}

  In this section we shall provide two equivalent characterizations for
  the strong calmness of $\mathcal{S}_{\rm KKT}$. First of all, we
  show that the strong calmness of $\mathcal{S}_{\rm KKT}$ at the origin
  for $(\overline{x},\overline{\lambda})\in\mathcal{S}_{\rm KKT}(0,0)$
  is equivalent to the following error bound for the KKT system.
 \begin{property}\label{Error-bound}
  ({\bf Error bound for KKT system}) Let $(\overline{x},\overline{\lambda})$ be
  a KKT point of \eqref{prob} with $(a,b)=(0,0)$. There exist
  $\varepsilon>0$ and a constant $c>0$ such that for all
  $(x,\lambda)\in\mathbb{B}_{\varepsilon}((\overline{x},\overline{\lambda}))$,
  \[
    \|x-\overline{x}\|+{\rm dist}(\lambda,\mathcal{M}(\overline{x},0,0))
    \le c\left\|\left(\begin{matrix}
                          \nabla_xL(x,\lambda)\\
                          g(x)-\Pi_{\mathcal{K}}\big(\lambda+g(x)\big)
                          \end{matrix}\right)\right\|.
  \]
 \end{property}
 From \cite{Izmailov12,Izmailov15,Cui16} we know that this local error bound
 plays a crucial role in analyzing the fast convergence rate of the sSQP and
 the augmented Lagrangian method.
 \begin{theorem}\label{Chara1-ULip}
  Let $(\overline{x},\overline{\lambda})$ be a KKT point of \eqref{prob}
  with $(a,b)=(0,0)$. Property \ref{Error-bound} holds at $(\overline{x},\overline{\lambda})$
  if and only if $\mathcal{S}_{\rm KKT}$ has the strong calmness
  at the origin for $(\overline{x},\overline{\lambda})$.
 \end{theorem}
 \begin{proof}
  ``$\Longrightarrow$.'' Suppose that Property \ref{Error-bound} holds
  at $(\overline{x},\overline{\lambda})$. Then, there exist $\varepsilon>0$
  and a constant $c>0$ such that for all
  $(x,\lambda)\in\mathbb{B}_{\varepsilon}((\overline{x},\overline{\lambda}))$,
  \[
    \|x-\overline{x}\|+{\rm dist}(\lambda,\mathcal{M}(\overline{x},0,0))
    \le c\left\|\left(\begin{matrix}
                          \nabla_xL(x,\lambda)\\
                          g(x)-\Pi_{\mathcal{K}}\big(\lambda+g(x)\big)
                          \end{matrix}\right)\right\|.
  \]
  Fix an arbitrary $\delta>0$. Take an arbitrary $(a,b)\in\mathbb{X}\times\mathbb{Y}$
  with $\|(a,b)\|\le\delta$. Let $(x',\lambda')$ be an arbitrary point from
  $\mathcal{S}_{\rm KKT}(a,b)\cap\mathbb{B}_{\varepsilon}((\overline{x},\overline{\lambda}))$.
  From the last inequality, we have
  \begin{equation}\label{error-ineq}
    \|x'-\overline{x}\|+{\rm dist}(\lambda',\mathcal{M}(\overline{x},0,0))
    \le c\left\|\left(\begin{matrix}
                          \nabla_xL(x',\lambda')\\
                          g(x')-\Pi_{\mathcal{K}}\big(\lambda'+g(x')\big)
                          \end{matrix}\right)\right\|.
  \end{equation}
  Since $(x',\lambda')\in\mathcal{S}_{\rm KKT}(a,b)$, we have
  $\nabla_xL(x',\lambda')=a$ and $g(x')-\Pi_{\mathcal{K}}(g(x')+\lambda'-b)=b$.
  Together with the inequality \eqref{error-ineq} and the global Lipschitz continuity
  of $\Pi_{\mathcal{K}}$,
  \begin{align*}
    \|x'-\overline{x}\|+{\rm dist}(\lambda',\mathcal{M}(\overline{x},0,0))
    &\le c\left\|\left(\begin{matrix}
                          a\\ b+\Pi_{\mathcal{K}}(\lambda'\!+\!g(x')-b)-\Pi_{\mathcal{K}}(\lambda'\!+\!g(x'))
                     \end{matrix}\right)\right\|\\
    &\le 2c\|(a,b)\|.
  \end{align*}
  From the arbitrariness of $(x',\lambda')$ in
  $\mathcal{S}_{\rm KKT}(a,b)\cap\mathbb{B}_{\varepsilon}((\overline{x},\overline{\lambda}))$,
  we conclude that the multifunction $\mathcal{S}_{\rm KKT}$ has
  the strong calmness at the origin for $(\overline{x},\overline{\lambda})$.

  \medskip
  \noindent
  ``$\Longleftarrow$.'' Suppose that $\mathcal{S}_{\rm KKT}$ has the strong calmness
  at the origin for $(\overline{x},\overline{\lambda})$. By Definition \ref{Upper-Lip},
  there exist $\delta'>0$, $\varepsilon'>0$ and $\kappa'>0$ such that
  for any $(a',b')\in\mathbb{X}\times\mathbb{Y}$ with $\|(a',b')\|\le\delta'$ and any
  $(x',\lambda')\in\mathcal{S}_{\rm KKT}(a',b')\cap\mathbb{B}_{\varepsilon'}((\overline{x},\overline{\lambda}))$,
  \begin{equation}\label{ULip-equa31}
    \|x'-\overline{x}\|+{\rm dist}(\lambda',\mathcal{M}(\overline{x},0,0))\le\kappa'\|(a',b')\|.
  \end{equation}
  Notice that the functions $\nabla_xL(x,\lambda)$ and $g(x)-\Pi_{\mathcal{K}}(g(x)+\lambda)$
  are continuous with respect to $(x,\lambda)$.
  There exist $\varepsilon>0$ and a constant $\gamma>0$
  such that for any $(x,\lambda)\in\mathbb{B}_{\varepsilon}((\overline{x},\overline{\lambda}))$,
  \begin{align*}
    \|\nabla_xL(x,\lambda)\|= \|\nabla_xL(x,\lambda)-\nabla_xL(\overline{x},\overline{\lambda})\|
    \le\frac{\delta'}{2\sqrt{2}},\qquad\qquad\qquad\qquad\\
    \|g(x)-\Pi_{\mathcal{K}}(g(x)\!+\!\lambda)\|
    =\|g(x)-\Pi_{\mathcal{K}}(g(x)\!+\!\lambda)-g(\overline{x})+\Pi_{\mathcal{K}}(g(\overline{x})\!+\!\overline{\lambda})\|
    \le\min\big(\frac{\delta'}{\sqrt{2}},\frac{\varepsilon'}{2}\big),\\
    \|\nabla g(x)\|\le \gamma\ \ {\rm and}\ \
    \big\|\nabla g(x)\big[g(x)-\Pi_{\mathcal{K}}(g(x)+\lambda)\big]\big\|\le\frac{\delta'}{2\sqrt{2}}.
    \qquad\qquad\qquad
  \end{align*}
  Set $\epsilon=\min({\varepsilon'}/{2},\varepsilon)$.
  Fix an arbitrary point pair $(x,\lambda)\in\mathbb{B}_{\epsilon}((\overline{x},\overline{\lambda}))$.
  Write $a=\nabla_xL(x,\lambda)$ and $b=g(x)-\Pi_{\mathcal{K}}(g(x)+\lambda)$.
  One may check that $(x,\lambda+b)\in S_{\rm KKT}(a+\!\nabla g(x)b,b)$.
  From $(x,\lambda)\in\mathbb{B}_{\epsilon}((\overline{x},\overline{\lambda}))$
  and the last inequalities, it follows that $\|(a+\!\nabla g(x)b,b)\|\le\delta'$.
  Also, from $(x,\lambda)\in\mathbb{B}_{\epsilon}\big((\overline{x},\overline{\lambda})\big)$
  and $\|b\|\le {\varepsilon'}/{2}$, we have
  \(
    \|(x,\lambda+b)-(\overline{x},\overline{\lambda})\|\le \varepsilon'.
  \)
  Now from $(x,\lambda+b)\in\mathbb{B}_{\varepsilon'}((\overline{x},\overline{\lambda}))
  \cap S_{\rm KKT}(a+\!\nabla g(x)b,b)$ and inequality \eqref{ULip-equa31},
  it follows that
  \[
    \|x-\overline{x}\|+{\rm dist}(\lambda+b,\mathcal{M}(\overline{x},0,0))\le
    \kappa'\left\|\left(\begin{matrix}
                   a+\!\nabla g(x)b\\ b
                   \end{matrix}\right)\right\|
  \]
  which, together with $a=\nabla_xL(x,\lambda)$ and $b=g(x)-\Pi_{\mathcal{K}}(g(x)+\lambda)$,
  implies that
  \begin{align*}
    \|x-\overline{x}\|+{\rm dist}(\lambda,\mathcal{M}(\overline{x},0,0))
    &\le \|x-\overline{x}\|+{\rm dist}(\lambda+b,\mathcal{M}(\overline{x},0,0))+\|b\|\\
    &\le \kappa'\sqrt{2(\|\nabla g(x)\|^2+1)}
         \left\|(a,b)\right\|+\|b\|\\
    &\le\big(1+\kappa'\sqrt{2\gamma^2+2}\big)
         \left\|\left(\begin{matrix}
                \nabla_xL(x,\lambda)\\ g(x)-\Pi_{\mathcal{K}}(g(x)+\lambda)
              \end{matrix}\right)\right\|.
  \end{align*}
  This, by the arbitrariness of $(x,\lambda)$ in
  $\mathbb{B}_{\epsilon}\big((\overline{x},\overline{\lambda})\big)$,
  shows that Property \ref{Error-bound} holds.
 \end{proof}
 \begin{remark}\label{remark-errbound}
  When $\mathcal{K}$ is specified as the nonpositive orthant cone
  in $\mathbb{R}^m$, the conclusion of Theorem \ref{Chara1-ULip}
  was obtained in \cite[Remark 1]{Izmailov12} and \cite[Remark 4]{Izmailov13}
  by invoking \cite[Theorem 2]{Fischer02}. In fact, from the proof of
  Theorem \ref{Chara1-ULip}, it is not difficult to obtain the following conclusion:
  the multifunction $\mathcal{S}_{\rm KKT}$ is calm at the origin for
  $(\overline{x},\overline{\lambda})$ if and only if there exist $\varepsilon>0$
  and a constant $c>0$ such that for all
  $(x,\lambda)\in\mathbb{B}_{\varepsilon}((\overline{x},\overline{\lambda}))$,
  \[
    {\rm dist}\big((x,\lambda),\mathcal{S}_{\rm KKT}(0,0)\big)
    \le c \left\|\left(\begin{matrix}
                          \nabla_xL(x,\lambda)\\
                          g(x)-\Pi_{\mathcal{K}}\big(\lambda+g(x)\big)
                          \end{matrix}\right)\right\|.
  \]
  This partly extends the result of \cite[Theorem 2]{Fischer02} to
  nonpolyhedral conic optimization.
  \end{remark}

  In order to provide the other equivalent characterization of
  the strong calmness, we need to introduce the concept
  of pseudo-isolated calmness of $\mathcal{X}_{\rm KKT}$.
 \begin{definition}\label{Pseudo-icalm}
  Let $\overline{x}$ be a stationary point of \eqref{prob} with $(a,b)=(0,0)$ and $\overline{\lambda}\in\mathcal{M}(\overline{x},0,0)$.
  The multifunction $\mathcal{X}_{\rm KKT}$ is said to have the pseudo-isolated
  calmness at the origin for $\overline{x}$ if there exist $\varepsilon>0,\delta>0$
  and a constant $\kappa>0$ such that for any $(a,b)\in\mathbb{X}\times\mathbb{Y}$
  with $\|(a,b)\|\le\delta$ and any $(x,\lambda)\in\mathcal{S}_{\rm KKT}(a,b)\cap\mathbb{B}_{\varepsilon}((\overline{x},\overline{\lambda}))$,
  the following estimate holds
  \[
    \|x-\overline{x}\|\le\kappa\|(a,b)\|.
  \]
 \end{definition}

  The reason why we call this property the pseudo-isolated calmness
  of $\mathcal{X}_{\rm KKT}$ is that it implies the isolated
  calmness of the multifunction $\mathcal{X}$ at $(\overline{\lambda},0,0)$
  for $\overline{x}$ by the following proposition, but may not imply
  the isolatedness of $\overline{x}$ unless
  $\mathcal{M}(\overline{x},0,0)=\{\overline{\lambda}\}$.
  In fact, when $\mathcal{M}(\overline{x},0,0)=\{\overline{\lambda}\}$,
  under Robinson's CQ, the pseudo-isolated calmness of $\mathcal{X}_{\rm KKT}$
  at the origin for $\overline{x}$ is equivalent to its isolated calmness
  at the origin for $\overline{x}$; see Appendix.
 \begin{proposition}\label{X-icalm}
  Let $\overline{x}$ be a stationary point of \eqref{prob}
  with $(a,b)=(0,0)$ and $\overline{\lambda}\in\mathcal{M}(\overline{x},0,0)$.
  \begin{itemize}
  \item [(a)] The multifunction $\mathcal{X}$ has the isolated calmness at $(\overline{\lambda},0,0)$
              for $\overline{x}$ if and only if there exist $\varepsilon>0$, $\delta>0$
              and $\kappa>0$ such that for any $(a,b)\in\mathbb{X}\times\mathbb{Y}$
              with $\|(a,b)\|\le\delta$ and any $(x,\lambda)\in\mathcal{S}_{\rm KKT}(a,b)\cap\mathbb{B}_{\varepsilon}((\overline{x},\overline{\lambda}))$,
              the following estimate holds
              \begin{equation}\label{X-estimate}
                \|x-\overline{x}\|\le\kappa\|(\lambda-\overline{\lambda},a,b)\|.
              \end{equation}

  \item [(b)] The multifunction $\mathcal{X}$ has the isolated calmness at $(\overline{\lambda},0,0)$
              for $\overline{x}$ if and only if $\xi=0$ is the unique solution to the following system
              \begin{equation}\label{Xsystem}
               \left\{\begin{array}{l}
               \nabla_{xx}^2L(\overline{x},\overline{\lambda})\xi=0,\\
                g'(\overline{x})\xi-\Pi_{\mathcal{K}}'(g(\overline{x})\!+\!\overline{\lambda};g'(\overline{x})\xi)=0.
               \end{array}\right.
              \end{equation}
 \end{itemize}
 \end{proposition}
 \begin{proof}
  {\bf(a)}``$\Longrightarrow$''. Suppose that the multifunction $\mathcal{X}$
  has the isolated calmness at $(\overline{\lambda},0,0)$ for $\overline{x}$.
  Then, there exist $\varepsilon'>0, \delta'>0$ and $\kappa'>0$ such that
  for any $(\lambda,a,b)\in\mathbb{B}_{\delta'}((\overline{\lambda},0,0))$,
  \[
    \mathcal{X}(\lambda,a,b)\cap\mathbb{B}_{\varepsilon'}(\overline{x}) \subseteq \{\overline{x}\}+\kappa'\|(\lambda-\overline{\lambda},a,b)\|\mathbb{B}_{\mathbb{Y}\times\mathbb{X}\times\mathbb{Y}}.
  \]
  Set $\delta=\delta'/2$ and $\varepsilon=\min(\varepsilon',\delta'/2)$.
  Fix an arbitrary $(a,b)\in\mathbb{X}\times\mathbb{Y}$ with $\|(a,b)\|\le\delta$.
  Pick up an arbitrary point $(x,\lambda)\in\mathcal{S}_{\rm KKT}(a,b)\cap\mathbb{B}_{\varepsilon}((\overline{x},\overline{\lambda}))$.
  Clearly, $(\lambda,a,b)\in\mathbb{B}_{\delta'}((\overline{\lambda},0,0))$
  and $x\in\mathcal{X}(\lambda,a,b)\cap\mathbb{B}_{\varepsilon'}(\overline{x})$.
  From the last inclusion, it immediately follows that
  \[
    \|x-\overline{x}\|\le\kappa'\|(\lambda-\overline{\lambda},a,b)\|.
  \]

  \noindent
  ``$\Longleftarrow$''. Suppose that there exist $\varepsilon>0,\delta>0$ and $\kappa>0$
  such that for any $(a,b)\in\mathbb{X}\times\mathbb{Y}$ with $\|(a,b)\|\le\delta$ and any
  $(x,\lambda)\in\mathcal{S}_{\rm KKT}(a,b)\cap\mathbb{B}_{\varepsilon}((\overline{x},\overline{\lambda}))$,
  the estimate in \eqref{X-estimate} holds.
  Set $\delta'=\min(\varepsilon/2,\delta)$ and $\varepsilon'=\varepsilon/2$.
  Fix an arbitrary $(\lambda,a,b)\in\mathbb{B}_{\delta'}((\overline{\lambda},0,0))$.
  Pick up an arbitrary $x\in\mathcal{X}(\lambda,a,b)\cap\mathbb{B}_{\varepsilon'}(\overline{x})$.
  Clearly, $(x,\lambda)\in\mathcal{S}_{\rm KKT}(a,b)\cap\mathbb{B}_{\varepsilon}((\overline{x},\overline{\lambda}))$.
  From \eqref{X-estimate},
  \(
    \|x-\overline{x}\|\le\kappa\|(\lambda-\overline{\lambda},a,b)\|,
  \)
  and hence $\mathcal{X}$ has the isolated calmness at $(\overline{\lambda},0,0)$
  for $\overline{x}$.

  \medskip
  \noindent
  {\bf(b)} From the definition of $\mathcal{X}$ and $\mathcal{S}_{\rm KKT}$,
  it is immediate to obtain $(\lambda,a,b,x)\in{\rm gph}\mathcal{X}$
  iff $\mathcal{A}(\lambda,a,b,x)\in {\rm gph}\mathcal{S}_{\rm KKT}$,
  where $\mathcal{A}:\mathbb{Y}\times\mathbb{X}\times\mathbb{Y}\times\mathbb{X}
  \to\mathbb{X}\times\mathbb{Y}\times\mathbb{X}\times\mathbb{Y}$ is defined by
  \[
    \mathcal{A}(\lambda,a,b,x)=(a,b,x,\lambda).
  \]
  Clearly, the linear map $\mathcal{A}$ is invertible
  and ${\rm gph}\mathcal{X}=\mathcal{A}^{-1}{\rm gph}\mathcal{S}_{\rm KKT}$.
  By \cite[Exercise 6.7]{RW98},
  \[
    \mathcal{T}_{{\rm gph}\mathcal{X}}(\overline{\lambda},0,0,\overline{x})
    =\mathcal{A}^{-1}\mathcal{T}_{{\rm gph}\mathcal{S}_{\rm KKT}}(0,0,\overline{x},\overline{\lambda}).
  \]
  Hence, $\Delta x\in D\mathcal{X}((\overline{\lambda},0,0)|\overline{x})
  (\Delta\lambda,\Delta a,\Delta b)$ iff $(\Delta x,\Delta\lambda)\in
  D\mathcal{S}_{\rm KKT}((0,0)|(\overline{x},\overline{\lambda}))(\Delta a,\Delta b)$.
  Together with the characterization of $D\mathcal{S}_{\rm KKT}((0,0)|(\overline{x},\overline{\lambda}))$
  in \cite[Lemma 18 \& 19]{DingSZ17}, it then follows that
  $\Delta x\in D\mathcal{X}((\overline{\lambda},0,0)|\overline{x})
  (\Delta\lambda,\Delta a,\Delta b)$ if and only if
  $(\Delta x,\Delta\lambda)$ satisfies
  \[
   \left\{\begin{array}{l}
   \nabla_{xx}^2L(\overline{x},\overline{\lambda})\Delta x+g'(\overline{x})\Delta\lambda=0,\\
   g'(\overline{x})\Delta x-\Pi_{\mathcal{K}}'(g(\overline{x})\!+\!\overline{\lambda};g'(\overline{x})\Delta x+\Delta\lambda)=0
   \end{array}\right.
  \]
  By Lemma \ref{chara-icalm}, we conclude that $\mathcal{X}$ is isolated calm
  at $(\overline{\lambda},0,0)$ for $\overline{x}$ if and only if
  system \eqref{Xsystem} has the unique solution $\xi=0$.
  The proof is then completed.
 \end{proof}

  It is worthwhile to mention that for the case that $\mathcal{K}$ is
  a semidefinite positive cone, Zhang and Zhang \cite{Zhang17} proved
  that the estimate in \eqref{X-estimate} holds iff \eqref{Xsystem}
  has only the trivial solution $\xi=0$. Here, we associate this property
  with the isolated calmness of $\mathcal{X}$.

  \medskip

  In addition, we also need the following lemma which states that
  the calmness of $\mathcal{M}$ at $(\overline{x},0,0)$
  for $\overline{\lambda}\in\mathcal{M}(\overline{x},0,0)$ is equivalent to
  that of $\mathcal{G}_{\overline{x}}$ at $(\nabla\!f(\overline{x}),g(\overline{x}))$
  for $\overline{\lambda}$.
  \begin{lemma}\label{icalm-MMGmap}
   Let $(\overline{x},\overline{\lambda})$ be a KKT point of the problem \eqref{prob}
   with $(a,b)\!=(0,0)$. Then, $\mathcal{M}$ is calm at $(\overline{x},0,0)$ for
   $\overline{\lambda}$ if and only if $\mathcal{G}_{\overline{x}}$
   is calm at $(\nabla\!f(\overline{x}),g(\overline{x}))$ for $\overline{\lambda}$.
  \end{lemma}
  \begin{proof}
   Suppose that the multiplier set mapping $\mathcal{M}$ is calm at $(\overline{x},0,0)$
   for $\overline{\lambda}$. Then, there exist $\varepsilon_1>0, \varepsilon_2>0$
   and $\kappa_1>0$ such that for all $(x,a,b)\in\mathbb{B}_{\varepsilon_1}((\overline{x},0,0))$,
   \[
     \mathcal{M}(x,a,b)\cap\mathbb{B}_{\varepsilon_2}(\overline{\lambda})
     \subseteq \mathcal{M}(\overline{x},0,0)
     +\kappa_1\|(x,a,b)-(\overline{x},0,0)\|\mathbb{B}_{\mathbb{X}\times\mathbb{X}\times\mathbb{Y}}.
   \]
   Write $\overline{\eta}=\nabla\!f(\overline{x})$
   and $\overline{y}=g(\overline{x})$. Fix an arbitrary $(\eta,y)\in\mathbb{B}_{\varepsilon_1}((\overline{\eta},\overline{y}))$.
   Pick up an arbitrary $\lambda\in\mathcal{G}_{\overline{x}}(\eta,y)
   \cap\mathbb{B}_{\varepsilon_2}(\overline{\lambda})$.
   It is easy to check that $\lambda\in \mathcal{M}(\overline{x},a,b)$ with
   $a=\overline{\eta}-\eta$ and $b=\overline{y}-y$. Clearly,
   $(\overline{x},a,b)\in\mathbb{B}_{\varepsilon_1}((\overline{x},0,0))$.
   From the last equation, it follows that
   \[
    {\rm dist}(\lambda,\mathcal{G}_{\overline{x}}(\overline{\eta},\overline{y}))
    ={\rm dist}(\lambda,\mathcal{M}(\overline{x},0,0))
    \le \kappa\|(a,b)-(0,0)\|=\|(\eta,y)-(\overline{\eta},\overline{y})\|.
   \]
   By the arbitrariness of $\lambda$ in $\mathcal{G}_{\overline{x}}(\eta,y)
   \cap\mathbb{B}_{\varepsilon_2}(\overline{\lambda})$,
   this shows that $\mathcal{G}_{\overline{x}}$ is calm at $(\overline{\eta},\overline{y})$
   for $\overline{\lambda}$.

   \medskip

   Conversely, suppose that the multifunction $\mathcal{G}_{\overline{x}}$
   is calm at $(\overline{\eta},\overline{y})$ for $\overline{\lambda}$.
   Then there exist $\delta_1>0,\delta_2>0$ and a constant $\nu>0$ such that
   for any $(\eta,y)\in\mathbb{B}_{\delta_1}((\overline{\eta},\overline{y}))$,
   \begin{equation}\label{equa0-Multi}
    \mathcal{G}_{\overline{x}}(\eta,y)\cap\mathbb{B}_{\delta_2}(\overline{\lambda})
    \subseteq \mathcal{G}_{\overline{x}}((\overline{\eta},\overline{y})
    +\nu\|(\eta,y)-(\overline{\eta},\overline{y})\|\mathbb{B}_{\mathbb{X}\times\mathbb{Y}}.
   \end{equation}
   Since the mappings $\nabla\!f, g$ and $\nabla g$ are locally Lipschitz
   continuous at $\overline{x}$, there exist $\widehat{\delta}_1>0$,
   $\widehat{\delta}_2>0$ and a constant $\widehat{\nu}>0$ such that
   for all $x\in\mathbb{B}_{\widehat{\delta}_1}(\overline{x})$
   and $\lambda\in\mathbb{B}_{\widehat{\delta}_2}(\overline{\lambda})$,
   \begin{equation}\label{Lip-fg}
     \|\nabla\!f(x)\!-\!\nabla\!f(\overline{x})\|+
      \|g(x)\!-\!g(\overline{x})\|+\|[\nabla\!g(x)-\nabla\!g(\overline{x})]\lambda\|
      \le \widehat{\nu}\|x-\overline{x}\|.
   \end{equation}
   Set $\varepsilon_1=\min(\frac{\delta_1}{\widehat{\nu}+1},\widehat{\delta}_1)$
   and $\varepsilon_2=\min(\widehat{\delta}_2,\delta_2)$.
   Fix an arbitrary $(x,a,b)\in\mathbb{B}_{\varepsilon_1}((\overline{x},0,0))$.
   Take an arbitrary point $\lambda\in \mathcal{M}(x,a,b)\cap
   \mathbb{B}_{\varepsilon_2}(\overline{\lambda})$.
   Then $\lambda\in\mathcal{G}_{\overline{x}}(\eta,y)$ with $y=g(x)-b$ and
   $\eta=\nabla f(x)\!-\!a+[\nabla\!g(x)\!-\!\nabla\!g(\overline{x})]\lambda$.
   Also, by using the inequality \eqref{Lip-fg} one may obtain that
   $\|(\eta,y)-(\overline{\eta},\overline{y})\|\le \widehat{\nu}\|x-\overline{x}\|+\|(a,b)\|\le\delta_1$.
   Thus, from \eqref{equa0-Multi} it follows that
   \begin{align*}
    {\rm dist}\big(\lambda,\mathcal{G}_{\overline{x}}(\overline{\eta},\overline{y})\big)
    &\le \!\nu\|(\eta,y)-(\overline{\eta},\overline{y})\|
    \le \nu(\widehat{\nu}\|x\!-\overline{x}\|+\!\|(a,b)\|)\\
    &\le\!\sqrt{2}\nu\max(\widehat{\nu},1)\|(x,a,b)\!-(\overline{x},0,0)\|.
   \end{align*}
   By the arbitrariness of $\lambda$ in $\mathcal{M}(x,a,b)\cap\mathbb{B}_{\varepsilon_2}(\overline{\lambda})$,
   $\mathcal{M}$ is calm at $(\overline{x},0,0)$ for $\overline{\lambda}$.
  \end{proof}
  \begin{theorem}\label{Chara2-ULip}
   Let $(\overline{x},\overline{\lambda})$ be a KKT point of \eqref{prob}
   with $(a,b)\!=(0,0)$ and $\mathcal{M}(\overline{x},0,0)\ne\{\overline{\lambda}\}$.
   Then, $\mathcal{S}_{\rm KKT}$ has the strong calmness at the origin
   for $(\overline{x},\overline{\lambda})$ if and only if $\mathcal{X}_{\rm KKT}$
   has the pseudo-isolated calmness at the origin for $\overline{x}$
   and $\mathcal{M}$ has the calmness at $(\overline{x},0,0)$ for $\overline{\lambda}$.
  \end{theorem}
  \begin{proof}
  ``$\Longleftarrow$''. Write $\overline{\eta}=\nabla f(\overline{x})$ and
  $\overline{y}=g(\overline{x})$. By Lemma \ref{icalm-MMGmap},
  $\mathcal{G}_{\overline{x}}$ is calm at $(\overline{\eta},\overline{y})$
  for $\overline{\lambda}$. Then, there exist $\delta'>0,\varepsilon'>0$
  and $\kappa'>0$ such that for all
  $(\eta,y)\in\mathbb{B}_{\varepsilon'}((\overline{\eta},\overline{y}))$,
  \begin{equation}\label{inclusion0}
    \mathcal{G}_{\overline{x}}(\eta,y)\cap\mathbb{B}_{\delta'}(\overline{\lambda})
    \subseteq\mathcal{G}_{\overline{x}}(\overline{\eta},\overline{y})
    +\kappa'\|(\eta,y)-(\overline{\eta},\overline{y})\|\mathbb{B}_{\mathbb{X}\times\mathbb{Y}}.
  \end{equation}
  Since $\mathcal{X}_{\rm KKT}$ has the pseudo-isolated calmness at the origin
  for $\overline{x}$, there exist $\widetilde{\delta}>0,\widetilde{\varepsilon}>0$
  and $\widetilde{\kappa}>0$ such that
  for any $(\widetilde{a},\widetilde{b})\in\mathbb{X}\times\mathbb{Y}$ with
  $\|(\widetilde{a},\widetilde{b})\|\le\widetilde{\delta}$ and
  any $(\widetilde{x},\widetilde{\lambda})\in\mathbb{B}_{\widetilde{\varepsilon}}((\overline{x},\overline{\lambda}))$,
  \begin{equation}\label{temp-estimate}
    \|x-\overline{x}\|\le \widetilde{\kappa}\|(\widetilde{a},\widetilde{b})\|.
  \end{equation}
  From the local Lipschitz continuity of $\nabla\!f(\cdot)$ and $g(\cdot)$,
  there exist $\widehat{\varepsilon}>0$ and constants $c_1>0$ and $c_2>0$ such that for all
  $(x,\lambda)\in\mathbb{B}_{\widehat{\varepsilon}}((\overline{x},\overline{\lambda}))$,
  the following inequalities hold:
  \begin{subnumcases}{}\label{eta-ineq1}
   \| \nabla\!f(x)+\nabla g(x)\lambda-\nabla g(\overline{x})\lambda-
   \nabla\!f(\overline{x})\|\le c_1\|x-\overline{x}\|,\\
   \label{eta-ineq2}
   \|g(x)-g(\overline{x})\|\le c_2\|x-\overline{x}\|,\
   \|g(x)-g(\overline{x})+b\|\le \frac{\varepsilon'}{\sqrt{2}},\,
   \|\lambda-\overline{\lambda}\|\le\delta',\\
   \label{eta-ineq3}
   \|\nabla f(x)+\nabla g(x)\lambda-\nabla g(\overline{x})\lambda
    -\nabla f(\overline{x})-a\|\le \frac{\varepsilon'}{\sqrt{2}}.
  \end{subnumcases}
  Set $\delta=\min(\varepsilon'/2,\widetilde{\delta})$ and
  $\varepsilon=\min(\widetilde{\varepsilon},\widehat{\varepsilon})$.
  Fix an arbitrary $(a,b)\in\mathbb{X}\times\mathbb{Y}$ with
  $\|(a,b)\|\le\delta$. Pick up an arbitrary $(x,\lambda)\in\mathcal{S}_{\rm KKT}(a,b)\cap\mathbb{B}_{\varepsilon}((\overline{x},\overline{\lambda}))$.
  We check that $\lambda\in \mathcal{G}_{\overline{x}}(\eta',y')$
  with $\eta'=\nabla\!f(x)+\nabla g(x)\lambda-a-\nabla g(\overline{x})\lambda$ and
  $y'=g(x)-b$. Clearly, $\lambda\in\mathbb{B}_{\delta'}(\overline{\lambda})$.
  Also, from the inequalities in \eqref{eta-ineq2}-\eqref{eta-ineq3},
  we have $(\eta',y')\in\mathbb{B}_{\varepsilon'}((\overline{\eta},\overline{y}))$.
  By \eqref{inclusion0} it follows that
  \begin{align*}
   {\rm dist}\big(\lambda,\mathcal{M}(\overline{x},0,0))
   &={\rm dist}\big(\lambda,\mathcal{G}_{\overline{x}}(\overline{\eta},\overline{y})\big)
    \le \kappa'\|(\eta',y')-(\nabla\!f(\overline{x}),g(\overline{x}))\|\\
   &=\kappa'\left\|\left(\begin{matrix}
                  \nabla\!f(x)+\nabla g(x)\lambda-a-\nabla g(\overline{x})\lambda-
                  \nabla\!f(\overline{x})\\
                  g(x)-b-g(\overline{x})
                  \end{matrix}\right)\right\|\\
   &\le \kappa'\Big[\sqrt{2c_1^2+2c_2^2}\|x-\overline{x}\|+\|(a,b)\|\Big]\\
   &\le \kappa'\Big[\widetilde{\kappa}\sqrt{2c_1^2+2c_2^2}\|(a,b)\|+\|(a,b)\|\Big]
  \end{align*}
  where the last equality is due to \eqref{temp-estimate} implied by
  $\|(a,b)\|\le\delta\le \widetilde{\delta}$ and
  $(x,\lambda)\in\mathbb{B}_{\widetilde{\varepsilon}}((\overline{x},\overline{\lambda}))$.
  This shows that $\mathcal{S}_{\rm KKT}$ has the strong calmness
  at the origin for $(\overline{x},\overline{\lambda})$.

  \medskip
  \noindent
  ``$\Longrightarrow$''. Suppose that $\mathcal{S}_{\rm KKT}$ has the strong calmness
  at the origin for $(\overline{x},\overline{\lambda})$. Clearly, $\mathcal{X}_{\rm KKT}$
  has the pseudo-isolated calmness at the origin for $\overline{x}$. It suffices to
  prove that $\mathcal{M}$ is calm at $(\overline{x},0,0)$ for $\overline{\lambda}$.
  By the strong calmness of $\mathcal{S}_{\rm KKT}$ at the origin
  for $(\overline{x},\overline{\lambda})$, there exist $\varepsilon>0,\delta>0$
  and $\kappa>0$ such that for any $(a,b)\in\mathbb{X}\times\mathbb{Y}$ with
  $\|(a,b)\|\le\delta$ and any $(x,\lambda)\in\mathcal{S}_{\rm KKT}(a,b)\cap\mathbb{B}_{\varepsilon}((\overline{x},\overline{\lambda}))$,
  the following estimate holds:
  \begin{equation}\label{xlambda}
   \|x-\overline{x}\|+{\rm dist}(\lambda,\mathcal{M}(\overline{x},0,0))\le\kappa\|(a,b)\|.
  \end{equation}
  Set $\varepsilon'=\frac{1}{\sqrt{2}}\min(\varepsilon,\delta)$ and $\delta'=\frac{1}{\sqrt{2}}\varepsilon$.
  Fix an arbitrary $(x,a,b)\in\mathbb{B}_{\varepsilon'}((\overline{x},0,0))$.
  Pick up an arbitrary $\lambda\in\mathcal{M}(x,a,b)\cap\mathbb{B}_{\delta'}(\overline{\lambda})$.
  Clearly, $(x,\lambda)\in\mathcal{S}_{\rm KKT}(a,b)\cap\mathbb{B}_{\varepsilon}((\overline{x},\overline{\lambda}))$.
  By \eqref{xlambda},
  \[
   {\rm dist}(\lambda,\mathcal{M}(\overline{x},0,0))\le\kappa\|(a,b)\|
   \le\kappa\|(x-\overline{x},a,b)\|.
  \]
  This shows that $\mathcal{M}$ is calm at $(\overline{x},0,0)$ for $\overline{\lambda}$.
  Thus, we complete the proof.
  \end{proof}

 \section{Pseudo-isolated calmness of $\mathcal{X}_{\rm KKT}$}\label{sec4}

 We shall focus on the characterizations of the pseudo-isolated
 calmness of $\mathcal{X}_{\rm KKT}$ in terms of the noncriticality of
 the associated multiplier. Along with Theorem \ref{Chara2-ULip},
 some sufficient characterizations are also obtained
 for the strong calmness of $\mathcal{S}_{\rm KKT}$.
 \begin{proposition}\label{necessity1-XKKT}
  Let $(\overline{x},\overline{\lambda})$ be a KKT point of the problem \eqref{prob}
  with $(a,b)\!=\!(0,0)$. If $\mathcal{X}_{\rm KKT}$ has the pseudo-isolated
  calmness at the origin for $\overline{x}$, then $\overline{\lambda}$ is noncritical.
  \end{proposition}
 \begin{proof}
  Let $\widetilde{\mathcal{S}}_{\rm KKT}(a,b)\!:=\big\{(x,\lambda)\ |\
  \nabla\!f(x)+\!\nabla g(x)(\lambda\!+\!b)=a,g(x)-\!\Pi_{\mathcal{K}}(g(x)\!+\!\lambda)=b\big\}$
  for $(a,b)\in\mathbb{X}\times\mathbb{Y}$. One may check that
  $(x,\lambda)\in\mathcal{S}_{\rm KKT}(a,b)$ iff
  $(x,\lambda-b)\in\widetilde{\mathcal{S}}_{\rm KKT}(a,b)$.
  By the given assumption and Definition \ref{Pseudo-icalm},
  there exist $\varepsilon'>0,\delta'>0$ and $\kappa'>0$ such that
  for any $(a,b)\in\mathbb{X}\times\mathbb{Y}$ with $\|(a,b)\|\le\delta'$
  and any $(x,\lambda)\in\widetilde{\mathcal{S}}_{\rm KKT}(a,b)\cap
  \mathbb{B}_{\varepsilon'}((\overline{x},\overline{\lambda}))$,
  \begin{equation}\label{estimation}
    \|x-\overline{x}\|\le \kappa'\|(a,b)\|.
  \end{equation}
  Let $(\xi^*,v^*)\in\mathbb{X}\times\mathbb{Y}$ be an arbitrary solution
  of system \eqref{noncritical-system}. For any sufficiently small $t>0$,
  define $x_t:=\overline{x}+t\xi^*$ and $\lambda_t:=\overline{\lambda}+tv^*$.
  By the directional differentiability of $\Pi_{\mathcal{K}}$,
  \begin{align*}
  \Pi_{\mathcal{K}}(g(x_t)+\lambda_t)
  &=\Pi_{\mathcal{K}}\big(g(\overline{x})+tg'(\overline{x})\xi^*+\overline{\lambda}+tv^*+o(t)\big)\\
  &=\Pi_{\mathcal{K}}(g(\overline{x})+\overline{\lambda})
    +t\Pi_{\mathcal{K}}'(g(\overline{x})\!+\!\overline{\lambda};g'(\overline{x})\xi^*+v^*)+o(t)\\
  &=g(\overline{x})+tg'(\overline{x})\xi^*+o(t),
  \end{align*}
  where the last equality is due to the fact that $(\xi^*,v^*)$ is a solution
  of the system \eqref{noncritical-system} and $g(\overline{x})=\Pi_{\mathcal{K}}(g(\overline{x})+\overline{\lambda})$.
  Together with $g(x_t)=g(\overline{x})+tg'(\overline{x})\xi^*+o(t)$,
  it follows that
  \[
    b_t:=g(x_t)-\Pi_{\mathcal{K}}(g(x_t)+\lambda_t)=o(t).
  \]
  In addition, from $\nabla f(\overline{x})+\nabla g(\overline{x})\overline{\lambda}=0$
  and $\nabla^2f(\overline{x})\xi^*+\nabla(\nabla g(\cdot)\overline{\lambda})(\overline{x})\xi^*
  =\nabla_{xx}^2L(\overline{x},\overline{\lambda})\xi^*$,
  \begin{align*}
   a_t&:=\nabla f(x_t)+\nabla g(x_t)(\lambda_t+b_t)\\
   &=\nabla f(\overline{x})+t\nabla^2f(\overline{x})\xi^*
    +(\nabla g(\overline{x})+t\nabla(\nabla g(\cdot))(\overline{x})\xi^*+o(t))(\overline{\lambda}+tv^*+b_t)+o(t)\\
   &=t\nabla_{xx}^2L(\overline{x},\overline{\lambda})\xi^*
     +t\nabla g(\overline{x})v^*+o(t)=o(t)
  \end{align*}
  where the last equality is using the fact that $(\xi^*,v^*)$ is
  the solution of \eqref{noncritical-system}. The last two equations show that,
  for all sufficiently small $t>0$, $(x_t,\lambda_t)\in\widetilde{\mathcal{S}}_{\rm KKT}(a_t,b_t)
  \cap\mathbb{B}_{\varepsilon'}(\overline{x},\overline{\lambda})$
  with $\|(a_t,b_t)\|\le\delta'$.
  From \eqref{estimation}, for all sufficiently small $t>0$, it holds that
  \[
    \|\xi^*\|\le \kappa'\|(a_t,b_t)\|/t\to 0.
  \]
  This implies $\xi^*=0$. By Proposition \ref{noncritical-prop},
  $\overline{\lambda}$ is noncritical for \eqref{prob}
  with $(a,b)=(0,0)$.
 \end{proof}

 Proposition \ref{necessity1-XKKT} shows that the pseudo-isolated calmness
 of $\mathcal{X}_{\rm KKT}$ implies the noncriticality of the associated
 multiplier. However, its converse conclusion generally does not
 hold unless $\mathcal{K}$ is polyhedral.
 Motivated by \cite[Theorem 3.2]{Cui16} and \cite[Theorem 3.3]{Zhang17},
 we establish the converse conclusion of Proposition \ref{necessity1-XKKT}
 under an additional condition.
 \begin{proposition}\label{sufficient-XKKT}
   Let $(\overline{x},\overline{\lambda})$ be a KKT point of \eqref{prob}
   with $(a,b)=(0,0)$. Define the sets
   \begin{equation}\label{Sigma}
   \begin{split}
    \Sigma(\overline{x},\overline{\lambda}):=\left\{(\xi,\zeta)\in\mathbb{X}\times\mathbb{Y}\!:\
    \nabla_{xx}^2L(\overline{x},\overline{\lambda})\xi
    +\nabla g(\overline{x})\big[\zeta+\frac{1}{2}\nabla\Upsilon(g'(\overline{x})\xi)\big]=0\right\},\\
    \Gamma(\overline{x},\overline{\lambda}):=
    \Big\{(\xi,\zeta)\in\mathbb{X}\times\mathbb{Y}\!:\ g'(\overline{x})\xi
    \in\mathcal{C}_{\mathcal{K}}(g(\overline{x}),\overline{\lambda}),
    \zeta\in[\mathcal{C}_{\mathcal{K}}(g(\overline{x}),\overline{\lambda})]^{\circ}\Big\}.\qquad
    \end{split}
    \end{equation}
   Suppose that $\langle g'(\overline{x})\xi,\zeta\rangle\ge0$ for any $(\xi,\zeta)\in\Sigma(\overline{x},\overline{\lambda})\cap\Gamma(\overline{x},\overline{\lambda})$
   and $\nabla g(\overline{x})[\mathcal{C}_{\mathcal{K}}(g(\overline{x}),\overline{\lambda})]^{\circ}$
   is closed. Then, under the noncriticality of the Lagrange multiplier $\overline{\lambda}$
   for the problem \eqref{prob} with $(a,b)=(0,0)$, $\mathcal{X}_{\rm KKT}$
   has the pseudo-isolated calmness at the origin for $\overline{x}$.
  \end{proposition}
  \begin{proof}
  Suppose on the contradiction that $\mathcal{X}_{\rm KKT}$ does not
  have the pseudo-isolated calmness at the origin for $\overline{x}$.
  By Definition \ref{Pseudo-icalm}, there exist sequences $\{(a^k,b^k)\}\to 0$
  and $\{(x^k,\lambda^k)\}\to(\overline{x},\overline{\lambda})$ with
   $(x^k,\lambda^k)\in\mathcal{S}_{\rm KKT}(a^k,b^k)$ for each $k\in\mathbb{N}$
   such that
   \(
     \frac{\|(a^k,b^k)\|}{\|x^k-\overline{x}\|}\to 0
   \)
  as $k\to\infty$. For convenience, we write $t_k:=\|x^k-\overline{x}\|$ for each $k$,
  and assume (if necessary taking a subsequence) that $\frac{x^k-\overline{x}}{t_k}\to\xi$
  for some $\xi\in\mathbb{X}$ with $\|\xi\|=1$ when $k\to\infty$.
  From $(x^k,\lambda^k)\in\mathcal{S}_{\rm KKT}(a^k,b^k)$
  and the definition of $\mathcal{S}_{\rm KKT}$, it immediately follows that
  \begin{equation}\label{xk-equa}
   \nabla f(x^k)+\nabla g(x^k)\lambda^k=a^k\ \ {\rm and}\ \
   \lambda^k\in\mathcal{N}_{\mathcal{K}}(g(x^k)-b^k).
  \end{equation}
  By using the twice continuous differentiability of $f$ and $g$,
   for each $k\in\mathbb{N}$ we have
  \begin{align*}
    \nabla f(x^k)&=\nabla f(\overline{x})+\nabla^2f(\overline{x})(x^k-\overline{x})+o(t_k),\\
    g'(x^k)&=g'(\overline{x})+g''(\overline{x})(x^k-\overline{x})+o(t_k).
  \end{align*}
  Together with the first equality in \eqref{xk-equa} and
  $\nabla f(\overline{x})+\nabla g(\overline{x})\overline{\lambda}=0$,
  we immediately obtain
  \begin{equation}\label{ak-xk}
   a^k=\nabla^2f(\overline{x})(x^k-\overline{x})
    +\nabla g(\overline{x})(\lambda^k-\overline{\lambda})
    +\big[g''(\overline{x})(x^k-\overline{x})]^*\lambda^k +o(t_k).
  \end{equation}
  By dividing the two sides of \eqref{ak-xk} with $t_k$
  and taking the limit $k\to\infty$, we have
  \begin{equation}\label{equa-ak0}
    w:=\lim_{k\to\infty}\frac{\nabla g(\overline{x})(\lambda^k-\overline{\lambda})}{t_k}
    =-\nabla_{xx}^2L(\overline{x},\overline{\lambda})\xi
  \end{equation}
  where the equality is using $\lim_{k\to\infty}{a^k}/{t_k}=0$
  and $\lim_{k\to\infty}\lambda^k=\overline{\lambda}$.
   Write $\overline{y}:=g(\overline{x})$. Since
  $g(x^k)=\overline{y}+t_kg'(\overline{x})\xi+o(t_k)$
  and $b^k=o(t_k)$, it then follows that
  \[
    g(x^k)-b^k=\overline{y}+t_ku^k\ \ {\rm with}\ \
    u^k:=g'(\overline{x})\xi+o(t_k)/t_k.
  \]
  Since $\lambda^k\in\mathcal{N}_{\mathcal{K}}(\overline{y}+t_ku^k)$ for each $k$,
  by the conic reducibility of $\mathcal{K}$ and Lemma \ref{lemma-reduction},
  for each sufficiently large $k$ there exist $\mu^k\!\in\mathcal{N}_D(\Xi(\overline{y}+t_ku^k))$
  such that $\lambda^k=\nabla\Xi(\overline{y}+t_ku^k)\mu^k$.
  Notice that $\nabla\Xi(\overline{y})\!:\mathbb{Y}\to\mathbb{Z}$ is injective.
  From $\lambda^k=\nabla\Xi(\overline{y}+t_ku^k)\mu^k$ and
  the convergence of $\lambda^k$, the sequence $\{\mu^k\}$ is bounded.
  We may assume (taking a subsequence if necessary) that $\mu^k\to\overline{\mu}$.
  Taking the limit $k\to\infty$ to the equality
  $\lambda^k=\nabla\Xi(\overline{y}+t_ku^k)\mu^k$ yields that
  $\overline{\lambda}=\nabla\Xi(\overline{y})\overline{\mu}$.
  Next, we proceed the arguments by three steps as shown below.

  \medskip
  \noindent
  {\bf Step 1: $g'(\overline{x})\xi\in\mathcal{C}_{\mathcal{K}}(g(\overline{x}),\overline{\lambda})$}.
  Since $g(x^k)-b^k\in\mathcal{K}$ and $g(\overline{x})\in\mathcal{K}$,
  we immediately have
  \begin{align*}
   \mathcal{T}_{\mathcal{K}}(g(\overline{x}))\ni
   g(x^k)-b^k-g(\overline{x})
   =g'(\overline{x})(x^k-\overline{x})+o(t_k),
  \end{align*}
  which implies $g'(\overline{x})\xi\in\mathcal{T}_{\mathcal{K}}(g(\overline{x}))$.
  Since
  \(
   0=\langle \mu^k,\Xi(\overline{y}+t_ku^k)\rangle
   =\langle\mu^k,t_k\Xi'(\overline{y})u^k+o(t_k)\rangle
  \)
  for all sufficiently large $k$, dividing the two sides of this equality
  by $t_k$ and taking the limit $k\to\infty$ yields that
  \(
    0=\langle \nabla\Xi(\overline{y})\overline{\mu},g'(\overline{x})\xi\rangle
    =\langle\overline{\lambda},g'(\overline{x})\xi\rangle.
  \)
 Together with $g'(\overline{x})\xi\in\mathcal{T}_{\mathcal{K}}(g(\overline{x}))$,
 we obtain that $g'(\overline{x})\xi\in\mathcal{C}_{\mathcal{K}}(g(\overline{x}),\overline{\lambda})$.
 The proof of this step is completed.

  \medskip
  \noindent
  {\bf Step 2: $w\!=\nabla\!g(\overline{x})v$ for some $v\in\mathbb{Y}$ with
  $v-\frac{1}{2}\nabla\Upsilon(g'(\overline{x})\xi)
  \in[\mathcal{C}_{\mathcal{K}}(g(\overline{x}),\overline{\lambda})]^{\circ}$.}
  Note that
  \begin{align*}
   \nabla g(\overline{x})\frac{\lambda^k-\overline{\lambda}}{t_k}
   &=\nabla g(\overline{x})\frac{\big(\nabla\Xi(\overline{y}+t_ku^k)-\nabla\Xi(\overline{y})\big)\mu^k
    +\nabla\Xi(\overline{y})(\mu^k-\overline{\mu})}{t_k}\nonumber\\
   &=\nabla g(\overline{x})\nabla(\nabla\Xi(\cdot)\mu^k)(\overline{y})u^k
     +\nabla g(\overline{x})\nabla\Xi(\overline{y})\frac{\mu^k-\overline{\mu}}{t_k}.
  \end{align*}
  By recalling $w=\lim_{k\to\infty}\frac{\nabla g(\overline{x})(\lambda^k-\overline{\lambda})}{t_k}$
  and taking the limit $k\to\infty$ to the both sides,
  \begin{equation}\label{temp-ak1}
    \lim_{k\to\infty}\nabla g(\overline{x})\nabla\Xi(\overline{y})\frac{\mu^k-\overline{\mu}}{t_k}
    =w-\nabla g(\overline{x})\nabla\big(\nabla\Xi(\cdot)\overline{\mu}\big)(\overline{y})(g'(\overline{x})\xi).
  \end{equation}
  We next argue that
  \(
   \nabla\Xi(\overline{y})\frac{\mu^k-\overline{\mu}}{t_k}
  \in [\mathcal{C}_{\mathcal{K}}(\overline{y},\overline{\lambda})]^{\circ}
  \)
  for each $k\in\mathbb{N}$. Fix an arbitrary $k\in\mathbb{N}$.
  Take an arbitrary $d\in\mathcal{C}_{\mathcal{K}}(\overline{y},\overline{\lambda})$.
  Then $\langle d,\overline{\lambda}\rangle=0$ by recalling that
  $\mathcal{C}_{\mathcal{K}}(\overline{y},\overline{\lambda})
  =\mathcal{T}_{\mathcal{K}}(\overline{y})\cap[\![\overline{\lambda}]\!]^{\perp}$.
  Also, by Lemma \ref{lemma-reduction}, $\mathcal{K}\cap\mathcal{Y}=\Xi^{-1}(D)$.
  Together with the surjectivity of $\Xi'(\overline{y})\!:\mathbb{Y}\to\mathbb{Z}$
  and \cite[Exercise 6.7]{RW98},
  $d\in\mathcal{T}_{\mathcal{K}}(\overline{y})=[\Xi'(\overline{y})]^{-1}\mathcal{T}_{D}(\Xi(\overline{y}))
  =[\Xi'(\overline{y})]^{-1}D$, and consequently
  \[
   \langle d, \nabla\Xi(\overline{y})(\mu^k\!-\overline{\mu})\rangle
   =\langle \Xi'(\overline{y})d,\mu^k\rangle
    -\langle d,\overline{\lambda}\rangle
   =\langle \Xi'(\overline{y})d,\mu^k\rangle\le 0,
  \]
  where the inequality is using $\Xi'(\overline{y})d\in D$
  and $\mu^k\in D^\circ$. So, the stated inclusion holds.
  From the given assumption, the set $\nabla\!g(\overline{x})[\mathcal{C}_{\mathcal{K}}(\overline{y},\overline{\lambda})]^{\circ}$
  is closed. Then, from \eqref{temp-ak1} there exists $\eta\in[\mathcal{C}_{\mathcal{K}}(\overline{y},\overline{\lambda})]^{\circ}$
  such that $\lim_{k\to\infty}\nabla g(\overline{x})\nabla\Xi(\overline{y})\frac{\mu^k-\overline{\mu}}{t_k}
  =\nabla g(\overline{x})\eta$. Along with \eqref{temp-ak1},
  \[
    w=\nabla g(\overline{x})v\ \ {\rm with}\ \
    v-\nabla\big(\nabla\Xi(\cdot)\overline{\mu}\big)(\overline{y})(g'(\overline{x})\xi)
    =\eta\in[\mathcal{C}_{\mathcal{K}}(\overline{y},\overline{\lambda})]^{\circ}.
  \]
  Recall that $\Upsilon(h)=\langle\overline{\mu},\Xi''(\overline{y})(h,h)\rangle$
  for $h\in\mathcal{C}_{\mathcal{K}}(\overline{y},\overline{\lambda})$.
  From $g'(\overline{x})\xi\in\mathcal{C}_{\mathcal{K}}(\overline{y},\overline{\lambda})$,
  it follows that $\nabla\big(\nabla\Xi(\cdot)\overline{\mu}\big)(\overline{y})(g'(\overline{x})\xi)=
  \frac{1}{2}\nabla\Upsilon(g'(\overline{x})\xi)$.
  Thus, we complete the proof of this step.

  \medskip
  \noindent
  {\bf Step 3: $\langle g'(\overline{x})\xi,v\rangle=\Upsilon(g'(\overline{x})\xi)$.}
  Notice that $v=\frac{1}{2}\nabla\Upsilon(g'(\overline{x})\xi)+\eta$
  with $\eta\in[\mathcal{C}_{\mathcal{K}}(\overline{y},\overline{\lambda})]^{\circ}$
  by Step 2. Also, from Step 2 and equation \eqref{equa-ak0},
  $\nabla_{xx}^2L(\overline{x},\overline{\lambda})\xi+\nabla\!g(\overline{x})v=0$,
  which means that $(\xi,\eta)\in\Sigma(\overline{x},\overline{\lambda})$.
  Recall from Step 1 that $g'(\overline{x})\xi\in\mathcal{C}_{\mathcal{K}}(\overline{y},\overline{\lambda})$.
  Then, $\langle g'(\overline{x})\xi,\eta\rangle\le 0$. Together with
  the given assumption, we also have $\langle g'(\overline{x})\xi,\eta\rangle\ge 0$.
  Thus, $\langle \eta,g'(\overline{x})\xi\rangle=0$,
  and consequently $\langle g'(\overline{x})\xi,v\rangle=\langle g'(\overline{x})\xi,
  \frac{1}{2}\nabla\Upsilon(g'(\overline{x})\xi)\rangle=\Upsilon(g'(\overline{x})\xi)$.

  \medskip

  So far, we have established that
  $g'(\overline{x})\xi\in\mathcal{C}_{\mathcal{K}}(\overline{y},\overline{\lambda})$,
  $v-\!\frac{1}{2}\nabla\Upsilon(g'(\overline{x})\xi)
  \in[\mathcal{C}_{\mathcal{K}}(\overline{y},\overline{\lambda})]^{\circ}$
  and $\langle g'(\overline{x})\xi,v\rangle=\Upsilon(g'(\overline{x})\xi)$.
  By Lemma \ref{dir-proj}, this is equivalent to saying that
  \[
    g'(\overline{x})\xi-\Pi_{\mathcal{K}}'(g(\overline{x})+\overline{\lambda};
    g'(\overline{x})\xi+v)=0.
  \]
  Together with $\nabla_{xx}^2L(\overline{x},\overline{\lambda})\xi+\nabla\!g(\overline{x})v=0$,
  it follows that $(\xi,v)$ satisfies the system \eqref{noncritical-system}.
  Since $\overline{\lambda}$ is noncritical, we obtain $\xi=0$,
  a contradiction to $\|\xi\|=1$.
 \end{proof}

  From Proposition \ref{calm-MGx} and Lemma \ref{icalm-MMGmap},
  the condition ${\rm ri}(\mathcal{N}_{\mathcal{K}}(g(\overline{x})))
  \cap\mathcal{H}(\nabla\!f(\overline{x}))\ne\emptyset$ or equivalently
  ${\rm ri}(\mathcal{N}_{\mathcal{K}}(g(\overline{x})))\cap\mathcal{M}(\overline{x},0,0)\ne\emptyset$
  implies the calmness of $\mathcal{M}$ at $(\overline{x},0,0)$ for
  $\overline{\lambda}\in\mathcal{M}(\overline{x},0,0)$. While the following
  lemma states that if the system $g(x)\in\mathcal{K}$ is metrically subregular
  at $\overline{x}$, this condition also implies the closedness
  of $\nabla\!g(\overline{x})[\mathcal{C}_{\mathcal{K}}(g(\overline{x}),\overline{\lambda})]^{\circ}$.
  \begin{lemma}\label{closed-lemma}
   Let $(\overline{x},\overline{\lambda})$ be a KKT point of the problem \eqref{prob}
   with $(a,b)=(0,0)$. Suppose that the multifunction $\mathcal{F}(\cdot)\!:=g(\cdot)-\mathcal{K}$
   is metrically subregular at $\overline{x}$ for the origin.
   If ${\rm ri}(\mathcal{N}_{\mathcal{K}}(g(\overline{x})))\cap\mathcal{M}(\overline{x},0,0)\ne\emptyset$,
   then the radial cone $\mathcal{R}_{\mathcal{N}_{\Omega}(\overline{x})}(-\nabla\!f(\overline{x}))$
   with $\Omega:=g^{-1}(\mathcal{K})$ is closed, which in turn implies the closedness of the set
   $\nabla g(\overline{x})[\mathcal{C}_{\mathcal{K}}(g(\overline{x}),\overline{\lambda})]^{\circ}$.
  \end{lemma}
  \begin{proof}
  Notice that $\Omega$ is the feasible set of \eqref{prob} with $(a,b)=(0,0)$.
  Since $\mathcal{F}$ is metrically subregular at $\overline{x}$ for $0$,
  from \cite[Corollary 2.1]{LiuPan17} we have
  $\mathcal{N}_{\Omega}(\overline{x})
  =\nabla\!g(\overline{x})\mathcal{N}_{\mathcal{K}}(g(\overline{x}))$.
  From ${\rm ri}(\mathcal{N}_{\mathcal{K}}(g(\overline{x})))\cap\mathcal{M}(\overline{x},0,0)\ne\emptyset$,
  there exists $\widehat{\lambda}\in {\rm ri}(\mathcal{N}_{\mathcal{K}}(g(\overline{x})))$
  such that
  \begin{equation}\label{temp-equa40}
   \nabla\!f(\overline{x})+\nabla g(\overline{x})\widehat{\lambda}=0.
  \end{equation}
  By the convexity of $\mathcal{N}_{\mathcal{K}}(g(\overline{x}))$,
  it follows that
  \(
     \mathcal{N}_{\mathcal{K}}(g(\overline{x}))+\mathbb{R}_{+}\widehat{\lambda}
     \subseteq\mathcal{N}_{\mathcal{K}}(g(\overline{x})).
   \)
  Then, by following the arguments as for \cite[Proposition 2.1]{Gfrerer17},
  we have
  \begin{align*}
   \nabla g(\overline{x})\mathcal{N}_{\mathcal{K}}(g(\overline{x}))
     \!+[\![\nabla\!g(\overline{x})\widehat{\lambda}]\!]
   &=\nabla g(\overline{x})\big(\mathcal{N}_{\mathcal{K}}(g(\overline{x}))+[\![\widehat{\lambda}]\!]\big)
   =\nabla g(\overline{x})\big(\mathcal{N}_{\mathcal{K}}(g(\overline{x}))-\mathbb{R}_{+}\widehat{\lambda}\big)\\
   &=\nabla g(\overline{x})\big[\mathbb{R}_{+}(\mathcal{N}_{\mathcal{K}}(g(\overline{x}))-\widehat{\lambda})\big]\\
   &={\rm cl}\big\{\nabla g(\overline{x})\big[\mathbb{R}_{+}(\mathcal{N}_{\mathcal{K}}(g(\overline{x}))
     -\widehat{\lambda})\big]\big\}\\
   &={\rm cl}\big\{\nabla g(\overline{x})\mathcal{N}_{\mathcal{K}}(g(\overline{x}))\!
     +[\![\nabla\!g(\overline{x})\widehat{\lambda}]\!] \big\}
  \end{align*}
   where the fourth equality is since $\mathbb{R}_{+}(\mathcal{N}_{\mathcal{K}}(g(\overline{x}))-\widehat{\lambda})$
   is a subspace parallel to the affine hull of $\mathcal{N}_{\mathcal{K}}(g(\overline{x}))$.
   Together with $\mathcal{N}_{\Omega}(\overline{x})
  =\nabla\!g(\overline{x})\mathcal{N}_{\mathcal{K}}(g(\overline{x}))$
  and \eqref{temp-equa40}, we obtain
  \begin{align*}
   \mathcal{R}_{\mathcal{N}_{\Omega}(\overline{x})}(-\nabla\!f(\overline{x}))
   &=\nabla g(\overline{x})\mathcal{N}_{\mathcal{K}}(g(\overline{x}))\!+[\![-\nabla\!f(\overline{x})]
   =\nabla g(\overline{x})\mathcal{N}_{\mathcal{K}}(g(\overline{x}))\!+[\![\nabla\!g(\overline{x})\widehat{\lambda}]\!]\\
   &={\rm cl}\big\{\nabla g(\overline{x})\mathcal{N}_{\mathcal{K}}(g(\overline{x}))\!
     +[\![\nabla\!g(\overline{x})\widehat{\lambda}]\!] \big\}\\
   &={\rm cl}\big\{\nabla g(\overline{x})\mathcal{N}_{\mathcal{K}}(g(\overline{x}))\!
     +[\![-\nabla\!f(\overline{x})]\!] \big\}
   ={\rm cl}\big(\mathcal{R}_{\mathcal{N}_{\Omega}(\overline{x})}(-\nabla\!f(\overline{x}))\big)
  \end{align*}
  where the first equality is due to \cite[Exercise 2.62]{BS00}
  and the closed convexity of $\mathcal{N}_{\Omega}(\overline{x})$.
  The last equality shows that the set
  $\mathcal{R}_{\mathcal{N}_{\Omega}(\overline{x})}(-\nabla\!f(\overline{x}))$ is closed.
  Notice that
  \begin{align}\label{key-inclusion}
    \nabla g(\overline{x})\mathcal{N}_{\mathcal{K}}(g(\overline{x}))\!+[\![-\nabla\!f(\overline{x}))]\!]
    &=\nabla g(\overline{x})\big(\mathcal{N}_{\mathcal{K}}(g(\overline{x}))+[\![\overline{\lambda}]\!]\big)
    \subseteq \nabla g(\overline{x})[\mathcal{C}_{\mathcal{K}}(g(\overline{x}),\overline{\lambda})]^{\circ}\nonumber\\
    &\subseteq {\rm cl}\big\{\nabla g(\overline{x})\mathcal{N}_{\mathcal{K}}(g(\overline{x}))
    +[\![-\nabla\!f(\overline{x}))]\!]\big\}\nonumber\\
    &=\nabla g(\overline{x})\mathcal{N}_{\mathcal{K}}(g(\overline{x}))\!+[\![-\nabla\!f(\overline{x}))]\!],
  \end{align}
  where the first equality is using the fact
  $\nabla f(\overline{x})+\nabla g(\overline{x}) \overline{\lambda}=0$,
  the first inclusion is by
  \(
   [\mathcal{C}_{\mathcal{K}}(g(\overline{x}),\overline{\lambda})]^{\circ}
   ={\rm cl}\big\{\mathcal{N}_{\mathcal{K}}(g(\overline{x}))+[\![\overline{\lambda}]\!]\big\},
  \)
  and the second inclusion is due to \cite[Theorem 6.6]{Roc70} and
  the convexity of $\mathcal{N}_{\mathcal{K}}(g(\overline{x}))$.
  This, along with the closedness of
  $\mathcal{R}_{\mathcal{N}_{\Omega}(\overline{x})}(-\nabla\!f(\overline{x}))$,
  shows that the set $\nabla g(\overline{x})[\mathcal{C}_{\mathcal{K}}(g(\overline{x}),\overline{\lambda})]^{\circ}$
  is closed. The proof is completed.
 \end{proof}
 \begin{remark}\label{remark40}
  {\bf(a)} Notice that the condition ${\rm ri}(\mathcal{N}_{\mathcal{K}}(g(\overline{x})))
  \cap\mathcal{M}(\overline{x},0,0)\ne\emptyset$ is much weaker than
  $\overline{\lambda}\in{\rm ri}(\mathcal{N}_{\mathcal{K}}(g(\overline{x})))$
  when the multiplier set $\mathcal{M}(\overline{x},0,0)$ is not a singleton.

  \medskip
  \noindent
  {\bf(b)} The metric subregularity of $\mathcal{F}(\cdot)=g(\cdot)-\mathcal{K}$
  at $\overline{x}$ for the origin is a very weak constraint qualification (CQ),
  which is clearly implied by Robinson's CQ since the latter is equivalent to
  the metric regularity of $\mathcal{F}$ at $\overline{x}$ for the origin by \cite{DR09}.
  For the research on the metric subregularity of the system $g(x)\in\mathcal{K}$,
  the reader may refer to \cite{Gfrerer16-MOR,Henrion05}.
 \end{remark}

  Now, by combining Lemma \ref{closed-lemma} with Proposition \ref{sufficient-XKKT}
  and Theorem \ref{Chara2-ULip}, we obtain the following sufficient characterization
  for the strong calmness of $\mathcal{S}_{\rm KKT}$.
  \begin{theorem}\label{Chara3-ULip}
  Let $(\overline{x},\overline{\lambda})$ be a KKT point of \eqref{prob}
  with $(a,b)\!=(0,0)$ and $\mathcal{M}(\overline{x},0,0)\ne\{\overline{\lambda}\}$.
  Suppose the multifunction $\mathcal{F}(\cdot)=g(\cdot)-\mathcal{K}$
  is metrically subregular at $\overline{x}$ for $0$. Define
  \[
    \widetilde{\Gamma}(\overline{x},\overline{\lambda}):=
    \Big\{(\xi,\zeta)\in\mathbb{X}\times\mathbb{Y}\!:\ g'(\overline{x})\xi
    \in\mathcal{C}_{\mathcal{K}}(g(\overline{x}),\overline{\lambda}),
    \nabla\!g(\overline{x})\zeta\in\nabla\!g(\overline{x})
    \big(\mathcal{N}_{\mathcal{K}}(g(\overline{x}))\!+[\![\overline{\lambda}]\!]\big)\Big\}.
  \]
  If $\langle g'(\overline{x})\xi,\zeta\rangle\ge0$
  for any $(\xi,\zeta)\in\Sigma(\overline{x},\overline{\lambda})\cap\widetilde{\Gamma}(\overline{x},\overline{\lambda})$
  and ${\rm ri}(\mathcal{N}_{\mathcal{K}}(g(\overline{x})))\cap\mathcal{M}(\overline{x},0,0)\ne\emptyset$,
  the noncriticality of $\overline{\lambda}$ is enough to the strong calmness
  of $\mathcal{S}_{\rm KKT}$ at $(0,0)$ for $(\overline{x},\overline{\lambda})$.
 \end{theorem}
 \begin{proof}
  The proof is same as that of Proposition \ref{sufficient-XKKT} except for
  the following fact
  \[
   \nabla g(\overline{x})\eta
   \in\nabla\!g(\overline{x})[\mathcal{C}_{\mathcal{K}}(g(\overline{x}),\overline{\lambda})]^{\circ}
    =\nabla\!g(\overline{x})
    \big(\mathcal{N}_{\mathcal{K}}(g(\overline{x}))\!+[\![\overline{\lambda}]\!]\big),
  \]
  where $\eta$ is from Step 2 in the proof of Proposition \ref{sufficient-XKKT},
  and the equality is by \eqref{key-inclusion}.
 \end{proof}
 \begin{remark}
  {\bf(a)} For the case where $\mathcal{K}$ is the semidefinite positive cone,
  the condition that $\langle g'(\overline{x})\xi,\zeta\rangle\ge0$
  for any $(\xi,\zeta)\in\Sigma(\overline{x},\overline{\lambda})\cap\widetilde{\Gamma}(\overline{x},\overline{\lambda})$
  is implied by the condition (ii) of \cite[Theorem 3.3]{Zhang17},
  and the closedness condition of
  $\nabla g(\overline{x})[\mathcal{C}_{\mathcal{K}}(g(\overline{x}),\overline{\lambda})]^{\circ}$
  there is removed. Thus, the result of Theorem \ref{Chara3-ULip} improves that of
 \cite[Theorem 3.3]{Zhang17} in this setting.

  \medskip
  \noindent
  {\bf(b)} When $\mathcal{M}(\overline{x},0,0)=\{\overline{\lambda}\}$,
  the strong calmness of $\mathcal{S}_{\rm KKT}$ at the origin
  for $(\overline{x},\overline{\lambda})$ becomes its isolated calmness
  at the origin for $(\overline{x},\overline{\lambda})$, which is equivalent to
  the calmness of $\mathcal{M}$ at $(\overline{x},0,0)$ for $\overline{\lambda}$
  along with the noncriticality of the multiplier $\overline{\lambda}$; see Appendix.
 \end{remark}

  Let $(\overline{x},\overline{\lambda})$ be a KKT point of \eqref{prob}
  with $(a,b)=(0,0)$. Recall that the second-order sufficient condition
  holds at $\overline{x}$ w.r.t. the multiplier $\overline{\lambda}$
  for \eqref{prob} with $(a,b)=(0,0)$ if
  \begin{equation}\label{SOSC-fix}
   \langle d,\nabla_{xx}^2L(\overline{x},\overline{\lambda})d\rangle
    -\sigma\big(\lambda,\mathcal{T}_{\mathcal{K}}^2(g(\overline{x}),g'(\overline{x})d)\big)> 0
    \quad\ \forall d\in\mathcal{C}(\overline{x})\backslash\{0\}.
  \end{equation}
  Clearly, this SOSC is stronger than the SOSC at $\overline{x}$ stated
  in \cite[Theorem 3.45 $\&$ 3.137]{BS00}.
  By Proposition \ref{noncritical-prop} and Lemma \ref{dir-proj},
  a simple argument by contradiction shows that the SOSC at $\overline{x}$
  w.r.t. $\overline{\lambda}$ implies the noncriticality of
  $\overline{\lambda}$. Then, under the conditions of Theorem \ref{Chara3-ULip},
  the SOSC at $\overline{x}$ w.r.t. $\overline{\lambda}$ is sufficient
  for the strong calmness of $\mathcal{S}_{\rm KKT}$.
  \section{Conclusions}

  In this paper, for a class of canonically perturbed conic programming,
  we have provided two equivalent characterizations for the strong calmness
  of the KKT solution mapping, i.e., this property is equivalent to
  a local error bound for solutions of perturbed KKT system, as well as
  the pseudo-isolated calmness of the stationary point mapping along with
  the calmness of the multiplier set mapping. In addition, some weaker sufficient
  conditions than \cite[Theorem 3.2]{Cui16} and \cite[Theorem 3.3]{Zhang17}
  for this property are also given in terms of the noncriticality of
  the Lagrange multiplier. The obtained results are crucial to achieve
  fast convergence of some algorithms such as the ALM for many important
  non-polyhedral convex conic optimization problems with degenerate solutions.

  \bigskip
  \noindent
  {\bf\large Appendix}
  \begin{aproposition}
  Let $(\overline{x},\overline{\lambda})$ be a KKT point of the problem \eqref{prob}
  with $(a,b)\!=\!(0,0)$. Then, the isolated calmness of $\mathcal{X}_{\rm KKT}$
  at the origin for $\overline{x}$ implies its pseudo-isolated calmness at
  the origin for $\overline{x}$. Also, the converse conclusion holds provided that
  Robinson's CQ holds at $\overline{x}$ for the problem \eqref{prob} with $(a,b)\!=\!(0,0)$
  and $\mathcal{M}(\overline{x},0,0)=\{\overline{\lambda}\}$.
 \end{aproposition}
 \begin{proof}
  The proof of the first part is easy, and we here focus on
  the proof of the second part. Suppose that $\mathcal{X}_{\rm KKT}$
  has the pseudo-isolated calmness at the origin for $\overline{x}$.
  If $\mathcal{X}_{\rm KKT}$ is not isolated calm at the origin for $\overline{x}$,
  then there exist the sequences $\{(a^k,b^k)\}\to 0$  and $\{x^k\}\to\overline{x}$
  with $x^k\in\mathcal{X}_{\rm KKT}(a^k,b^k)$ for each $k\in\mathbb{N}$ such that
   \(
     \frac{\|(a^k,b^k)\|}{\|x^k-\overline{x}\|}\to 0
   \)
  as $k\to\infty$. Since $x^k\in\mathcal{X}_{\rm KKT}(a^k,b^k)$,
  for each $k\in\mathbb{N}$ there exists $\lambda^k\in\mathcal{M}(x^k,a^k,b^k)$
  such that $(x^k,\lambda^k)\in \mathcal{S}_{\rm KKT}(a^k,b^k)$.
  Together with the pseudo-isolated calmness of $\mathcal{X}_{\rm KKT}$
  at the origin for $\overline{x}$, there exists $\epsilon_0>0$ such that
  for all sufficiently large $k$, $\|\lambda^k-\overline{\lambda}\|\ge\epsilon_0$.
  On the other hand, since Robinson's CQ holds at $\overline{x}$ for
  \eqref{prob} with $(a,b)\!=\!(0,0)$, the multifunction $\mathcal{M}$
  is locally bounded at $(\overline{x},0,0)$ in the sense of
  \cite[Definition 5.14]{RW98} (see also the proof of \cite[Theorem 3.2]{Robinson82}).
  Thus, the sequence $\{\lambda^k\}$ is bounded.
  Without loss of generality, assume that $\lambda^k\to\widehat{\lambda}$.
  Since $(x^k,\lambda^k)\in \mathcal{S}_{\rm KKT}(a^k,b^k)$, for each $k\in\mathbb{N}$
  it holds that
  \[
   \nabla\!f(x^k)\!+\nabla\!g(x^k)\lambda^k=a^k\ \ {\rm and}\ \
   g(x^k)-\Pi_{\mathcal{K}}(g(x^k)-b^k\!+\!\lambda^k)=b^k
   \]
   Taking the limit $k\to \infty$ on the last two equalities, we obtain
   $\widehat{\lambda}\in \mathcal{M}(\overline{x},0,0)=\{\overline{\lambda}\}$.
   This yields a contradiction to the result that
   $\|\lambda^k-\overline{\lambda}\|\ge\epsilon_0$ for all sufficiently large $k$.
   \end{proof}
  \begin{aproposition}\label{icalm-SKKTM}
   Let $(\overline{x},\overline{\lambda})$ be a KKT point of \eqref{prob} with $(a,b)=(0,0)$.
   The multifunction $\mathcal{S}_{\rm KKT}$ is isolated calm at $(0,0)$
   for $(\overline{x},\overline{\lambda})$ if and only if $\mathcal{M}$ is
   isolated calm at $(\overline{x},0,0)$ for $\overline{\lambda}$
   and the multiplier $\overline{\lambda}$ is noncritical for
   \eqref{prob} with $(a,b)=(0,0)$.
  \end{aproposition}
  \begin{proof}
  By \cite[Lemma 18\,$\&$ 19]{DingSZ17}, $\mathcal{S}_{\rm KKT}$ is isolated calm
  at the origin for $(\overline{x},\overline{\lambda})$ iff system
  \begin{equation*}
    \left\{\begin{array}{ll}
    \nabla_{xx}^2L(\overline{x},\overline{\lambda})u+\nabla\!g(\overline{x})v=0,\\
    g'(\overline{x})u-\Pi_{\mathcal{K}}'(g(\overline{x})\!+\!\overline{\lambda};g'(\overline{x})u\!+\!v)=0
   \end{array}\right.
  \end{equation*}
  has only the trivial solution $(u,v)=(0,0)$. From the proof of Lemma \ref{icalm-MMGmap},
  it is easy to obtain that $\mathcal{M}$ is isolated calm at $(\overline{x},0,0)$
  for $\overline{\lambda}$ iff $\mathcal{G}_{\overline{x}}$ is isolated
  calm at $(\nabla\!f(\overline{x}),\nabla g(\overline{x}))$ for $\overline{\lambda}$.
  Thus, together with Proposition \ref{noncritical-prop}, it suffices to
  argue that the isolated calmness of $\mathcal{G}_{\overline{x}}$ at
  $(\nabla\!f(\overline{x}),\nabla g(\overline{x}))$ for $\overline{\lambda}$
  is equivalent to saying $\Delta\lambda=0$ is the unique solution of
   \begin{equation}\label{MGcalm-system0}
     \left\{\begin{array}{ll}
      \nabla\!g(\overline{x})\Delta\lambda=0,\\
     \Pi_{\mathcal{K}}'(g(\overline{x})+\!\overline{\lambda};\Delta\lambda)=0.
     \end{array}\right.
   \end{equation}
   Indeed, it is not hard to check that
   ${\rm gph}\,\mathcal{G}_{\overline{x}}=\mathcal{A}^{-1}\big({\rm gph}\,\Phi^{-1}\big)$
   with $\Phi$ and $\mathcal{A}$ defined by
   \[
     \Phi(\lambda):=\!\left(\begin{matrix}
                     \nabla\!g(\overline{x})\lambda\\
                      -\Pi_{\mathcal{K}}(\lambda)
                      \end{matrix}\right)\ \ {\rm and}\ \
     \mathcal{A}(\eta,y,\lambda):=\!\left(\begin{matrix}
                -\eta\!+\!\nabla\!g(\overline{x})y\\ -y \\ \lambda+y
               \end{matrix}\right).
   \]
   Write $\overline{\eta}=\nabla\!f(\overline{x})$ and $\overline{y}=g(\overline{x})$.
   Since $\mathcal{A}$ is bijective, by \cite[Exercise 6.7]{RW98} we have
   \[
     \mathcal{T}_{{\rm gph}\,\mathcal{G}_{\overline{x}}}(\overline{\eta},\overline{y},\overline{\lambda})
     =\mathcal{A}^{-1}\!\left[\mathcal{T}_{{\rm gph}\,\Phi^{-1}}
     \big(\!-\!\overline{\eta}+\!\nabla\!g(\overline{x})\overline{y},-\overline{y},
     \overline{\lambda}+\overline{y}\big)\right].
   \]
   This, along with the definition of graphical derivative, gives the following equivalence:
   \begin{equation*}
    \Delta\lambda\in D\mathcal{G}_{\overline{x}}((\overline{\eta},\overline{y})|\,\overline{\lambda})(0,0)
    \Longleftrightarrow \Delta\lambda\in D\Phi^{-1}
    \big((-\overline{\eta}\!+\!\nabla\!g(\overline{x})\overline{y},-\overline{y})|\,\overline{y}\!+\!\overline{\lambda}\big)(0,0).
   \end{equation*}
   Then,
   \(
     D\mathcal{G}_{\overline{x}}((\overline{\eta},\overline{y})|\,\overline{\lambda})(0,0)=\!\{0\}
   \)
   if and only if
   \(
     D\Phi^{-1}\big((-\overline{\eta}+\!\nabla\!g(\overline{x})\overline{y},-\overline{y})|\,\overline{y}+\overline{\lambda}\big)(0,0)
     =\{0\}.
   \)
   By Lemma \ref{chara-icalm}, $\mathcal{G}_{\overline{x}}$ is isolated calm at $(\overline{\eta},\overline{y})$
   for $\overline{\lambda}$ if and only if $\Phi^{-1}$ is isolated calm
   at $(-\overline{\eta}+\!\nabla\!g(\overline{x})\overline{y},-\overline{y})$ for $\overline{y}+\overline{\lambda}$.
   In addition, by virtue of the Lipschitz continuity and the directional differentiability of $\Phi$,
   the following equivalence holds:
  \begin{align*}
   \Phi'(\overline{y}\!+\!\overline{\lambda};\Delta\lambda)
     =\left(\begin{matrix}
              \Delta\eta\\ \Delta y
      \end{matrix}\right)
   &\Longleftrightarrow (\Delta\eta,\Delta y)\in D\Phi(\overline{y}+\!\overline{\lambda}\,|(-\overline{\eta}+\nabla\!g(\overline{x})\overline{y},-\overline{y}))(\Delta\lambda)\\
   &\Longleftrightarrow
   \Delta\lambda\in D\Phi^{-1}((-\overline{\eta}+\!\nabla\!g(\overline{x})\overline{y},-\overline{y})|\,\overline{y}\!+\!\overline{\lambda})
     (\Delta\eta,\Delta y).
  \end{align*}
  Thus, using Lemma \ref{chara-icalm} again delivers the desired statement.
  \end{proof}
  \end{document}